\documentclass[11pt]{amsart}
\usepackage[latin1]{inputenc}

\usepackage{a4wide}
\usepackage{color}
\usepackage{bbm}
\usepackage{tikz-cd}
\usepackage{esint}
\usepackage{amsmath,amsthm,amssymb}
\usepackage{mathtools}
\usepackage{amsfonts}
\usepackage{amsmath,amsthm,amssymb,amscd}
\usepackage{latexsym}
\usepackage{graphicx}
\usepackage{mathrsfs}
\usepackage[colorlinks,citecolor=red,linkcolor=red]{hyperref}
\usepackage{comment}
\usepackage{cancel}

\newcommand{\fr}{\penalty-20\null\hfill\(\blacksquare\)}

\usepackage{fancyhdr}
\usepackage{lipsum} % only for showing some sample text
\makeatletter \@addtoreset{equation}{section} \makeatother

\newcommand{\norm}[1]{\left\lVert#1\right\rVert}

\theoremstyle{plain}% default
\newtheorem{thm}{Theorem}[section]
\newtheorem{lem}[thm]{Lemma}
\newtheorem{defn}[thm]{Definition}
\newtheorem{definition}[thm]{Definition}
\newtheorem{prop}[thm]{Proposition}
\newtheorem{cor}[thm]{Corollary}
\newtheorem{rem}[thm]{Remark}
\newtheorem{ex}[thm]{Example}

\DeclareMathOperator{\dt}{dt}

\DeclareMathOperator{\Lip}{Lip}
\DeclareMathOperator{\card}{card}

\newcommand{\bbG}{\mathbb{G}}
\newcommand{\G}{\mathbb{G}}
\newcommand{\g}{\mathfrak{g}}
\newcommand{\bbR}{\mathbb{R}}
\newtheorem*{notat}{Notation}
\begin{document}

\title[Variational Problems Concerning sub-Finsler metrics in Carnot groups]
{Variational Problems Concerning sub-Finsler metrics in Carnot groups}

\author[F. Essebei]{Fares Essebei}
\address[Fares Essebei]{Dipartimento di Matematica,	Universit\`a degli Studi di Trento,	Via Sommarive 14, 38123, Povo (Trento), Italy}
\email{fares.essebei@unitn.it}

\author[E. Pasqualetto]{Enrico Pasqualetto}
\address[Enrico Pasqualetto]{Scuola Normale Superiore, Piazza dei Cavalieri 7, 56126 Pisa, Italy}
\email{enrico.pasqualetto@sns.it}

\date{\today}
\maketitle

\begin{abstract}
This paper is devoted to the study of geodesic distances defined on a subdomain of a given Carnot group, which are bounded both from above and from below by fixed multiples of the Carnot--Carath\'{e}odory distance.
We show that the uniform convergence (on compact sets) of these distances can be equivalently characterized in terms of $\Gamma$-convergence of several kinds of variational problems.
Moreover, we investigate the relation between the class of intrinsic distances, their metric derivatives
and the sub-Finsler convex metrics defined on the horizontal bundle.
\end{abstract}

\section{Introduction}
Since three decades, many classical metric problems in Riemannian geometry have been translated to the context of Sub-Riemannian geometry and, in particular, to Carnot groups $\mathbb{G}$, which possess a rich geometry. Indeed, they are connected and simply connected Lie groups whose associated Lie algebra admits a finite-step stratification (see \cite{ledonne, Monti, Pansu, BonLanUgu}).\\
One of these problems is devoted to the study of a particular class of geodesic 
distances, bounded from above and below, which was deeply studied in the Euclidean space and in the setting of Lipschitz manifolds (see \cite{Vent, DCP1, DCP2, DCP4}).
Our first purpose is to generalize this class to Carnot groups, in other words, we consider all geodesic maps 
$d: \Omega \times \Omega \rightarrow \mathbb{R}$ locally equivalent to the so-called 
Carnot--Carath\'{e}odory distance $d_{cc}$, defined on an open subset $\Omega \subset \mathbb{G}$
and we say that they belong to $\mathcal{D}_{cc}(\Omega)$ if:
\begin{align}
\frac{1}{\alpha} d_{cc}(x, y) \leq d(x, y) \leq \alpha 
d_{cc}(x, y)\qquad \forall x,y \in \Omega,
\end{align}
for some $\alpha \geq 1$.
According to \cite{Vent, DCP2}, it is quite natural to construct, on the so-called horizontal bundle $H \mathbb{G}$, the family of metrics $\varphi_d : H \mathbb{G} \rightarrow [0, + \infty)$ 
associated to $d \in \mathcal{D}_{cc}(\Omega)$ by differentiation:
\[
\varphi_d(x,v)\coloneqq\limsup_{t\searrow 0}
\frac{d\big(x,x\cdot\delta_t\,{\rm exp}({\rm d}_x\tau_{x^{-1}}[v])\big)}{t}
\quad\text{ for every }(x,v)\in H\G.
\]
Inspired by the notation for horizontal curves introduced by Scott D. Pauls in \cite{Pauls}, in the previous definition we denote with  $x \cdot \delta_t \exp({\rm d}_x\tau_{x^{-1}}[v])$ the dilation curve starting from the point $x \in \mathbb{G}$ with direction given by the left translation at the identity of a horizontal vector defined on the fiber $H_x \mathbb{G}$.
In particular, it turns out that $\varphi_d$ is a \textit{sub-Finsler convex metric}. These objects play an important role in the setting of the so-called \textit{Sub-Finsler Carnot groups} (for a reference see for example \cite[Section 6]{LeD21}), playing the same role of  Finsler metrics with the difference that they are defined only on the horizontal bundle.  
\begin{defn} $\varphi : H \mathbb{G}\rightarrow [0, + \infty)$ is a \emph{sub-Finsler convex metric} belonging to $\mathcal{M}_{cc}^\alpha(\mathbb{G})$ if
\begin{itemize}
\item[1)] \(\varphi:H\G\to\mathbb R\) is Borel measurable,
where \(H\G\) is endowed with the product $\sigma$-algebra;
\item[2)] $\varphi(x, \delta_{\lambda}^* v) = \lvert \lambda \rvert \varphi(x, v)$ for every $(x,v)\in H\mathbb{G}$ and $\lambda \in \mathbb{R}$;
\item[3)] $\frac{1}{\alpha} \| v\|_x \leq \varphi(x, v) \leq \alpha \| v \|_x$ for every $(x, v) \in H \mathbb{G}$;
\item[4)] $\varphi(x,v_1 + v_2)\leq \varphi(x,v_1)+\varphi(x,v_2)$ for every $x\in\mathbb{G}$ and $v_1, v_2\in H_x \mathbb{G}$.
\end{itemize}
In particular, \(\varphi(x,\cdot)\) is a norm on \(H_x\G\) for every \(x\in\G\).
\end{defn}

At the same time, it is also natural to consider the length functional 
induced by the metric derivative $\varphi_d$. We will show that
we can reconstruct the distance $d$ by minimizing the corresponding functional (Theorem \ref{lunghezza}).
\medskip

The first purpose of this paper is to compare the asymptotic behaviour of different kinds of functionals involving distances defined on a given open subset $\Omega$ of a Carnot group $\mathbb{G}$. 
Indeed, the most common approach in order to study the following variational problems relies on $\Gamma$-convergence of the corresponding functionals
(see Section 4).
In particular, inspired by the proof contained in \cite{BDF}, we state the equivalence between the 
$\Gamma$-convergence of the functionals $L_n$ and $J_n$ associated to a sequence of distances 
$(d_n)_{n \in \mathbb{N}}\subset \mathcal{D}_{cc}(\Omega)$ respectively through
$$L_n(\gamma) = \int_0^1 \varphi_{d_n}(\gamma(t), \dot{\gamma}(t)) \,{\rm d}t \quad \mbox{and} \quad J_n(\mu) = \int_{\Omega \times \Omega} d_n(x, y)\, {\rm d} \mu(x, y),$$
where $\gamma : [0,1]\rightarrow \Omega$ is a horizontal curve (Definition \ref{horizontalcurve}) and $\mu$ is a positive and finite Borel measure on $\Omega \times \Omega$.
This kind of result has been already studied in the literature,
especially for what concerns the homogenization of Riemannian and Finsler metrics (\cite{AB, AV}).
Moreover, we show an additional characterization when $\Omega$ is bounded (Theorem \ref{gammaconvergence}, point $(iv)$).
\medskip

The second purpose is to study a different application: the intrinsic analysis of sub-Finsler metrics. 
More precisely,
under suitable regularity assumptions on the metric under consideration,
we prove the following result (Theorem \ref{phi=Lip}):
\begin{equation}\label{p=L}
 \varphi(x, \nabla_{\mathbb{G}} f(x))= \Lip_{\delta_{\varphi}} f(x) 
\quad\text{for a.e.\ }  x \in \mathbb{G},
\end{equation}
where $f : \Omega \subset \mathbb{G} \rightarrow \mathbb{R}$ is a Pansu-differentiable function, $\delta_{\varphi}$ is the distance defined in (\ref{ugu}) below, $\varphi \in \mathcal{M}_{cc}^\alpha(\Omega)$ is a sub-Finsler convex metric,
and the pointwise Lipschitz constant of $f$ is given by
$$ \Lip_{\delta_{\varphi}} f(x) = \limsup_{y \rightarrow x} \frac{ \lvert f(y) - f(x) \rvert}{\delta_{\varphi}(x, y)}
\quad\mbox{for every }x\in\G.$$
The equality \eqref{p=L} may be regarded as a generalization of a result achieved in \cite{sturm}, and further generalized by Chang Y.\ Guo to admissible Finsler metrics defined 
on open subsets of $\mathbb{R}^n$, in \cite{GUO}.
Then, in order to prove (\ref{p=L}), we crucially observe that the quantity  
\begin{equation}\label{ugu}
\delta_{\varphi}(x, y) \coloneqq \sup\big\{ \lvert f(x) - f(y) \rvert \, \big| \,f\colon\G\rightarrow\mathbb R \mbox{ Lipschitz},\, 
\norm{\varphi(\cdot, \nabla_{\mathbb{G}} f (\cdot))}_{\infty} \leq 1 \big\}
\end{equation}
coincides with the intrinsic distance $d_{\varphi^\star}$, induced by the dual metric. 
This happens, for instance, when we assume that the sub-Finsler metric $\varphi$ is either lower semicontinuous or upper semicontinuous on the horizontal bundle (see Theorem \ref{maintheorem} and Corollary \ref{corollary}).
These results are a generalization of the analogous statement in \cite{DCP1}, due to De Cecco and Palmieri. 
The proof of Theorem \ref{maintheorem} heavily relies on two results contained in \cite{LDLP}.
The first one allows us to approximate an upper semicontinuous sub-Finsler metric with a family of Finsler metrics. 
The second result lets us approximate from below the sub-Finsler distance
with a family of induced Finsler distances.
Finally, we show that in many cases the distance $\delta_\varphi$, albeit defined as a supremum among Lipschitz functions, is actually already determined by smooth functions (cf.\ Proposition \ref{smooth}).
An important step in proving this fact is to approximate (say, uniformly on compact sets) any $1$-Lipschitz function with a sequence of smooth $1$-Lipschitz functions; here, the key point is that the Lipschitz constant is preserved. Since this approximation result holds in much greater generality (for instance, on possibly rank-varying sub-Finsler structures) and might be of independent interest, we will treat it in Appendix A.\\

We now give a short descriptive plan of the paper. In Section 2 we collect some of the basic geometric facts about Carnot groups and we present all preliminaries about horizontal bundles, horizontal curves, Pansu differentiability for Lipschitz functions, and the main concepts related to sub-Finsler convex metrics. In particular, we prove some properties about dual metrics associated to sub-Finsler metrics.
In Section 3 we introduce the main concept of metric derivative and we show its convexity on the horizontal fibers. Moreover, we prove a classical length representation theorem through a distance reconstruction result.
Section 4 contains the equivalence theorem between the uniform convergence of distances in $\mathcal{D}_{cc}(\Omega)$ and $\Gamma$-convergence of the length and energy functionals.
Finally, in Section 5 we introduce the main concepts of intrinsic distance and metric density and we prove the main theorems of this paper.
\medskip

\noindent\textbf{Acknowledgements.}
The authors are grateful to Andrea Pinamonti, Francesco Serra Cassano, and Gioacchino Antonelli for several discussions on the topics of this paper,
as well as to Davide Vittone for the useful comments about Theorem \ref{thm:smooth_approx}. The second named author is supported by the Balzan project
led by Luigi Ambrosio.
\section{Preliminaries}
\subsection{Carnot groups}
A connected and simply connected Lie group $\G$ is said to be a {\it Carnot group  of
step $k$} if its Lie algebra ${\mathfrak{g}}$  admits a {\it step $k$ stratification}, i.e.,
there exist linear subspaces $\mathfrak{g}_1, \ldots,\mathfrak{g}_k$ of $\mathfrak{g}$ such that
\begin{equation}\label{stratificazione}
\mathfrak{g}=\mathfrak{g}_1\oplus \ldots \oplus \mathfrak{g}_k,\quad [\mathfrak{g}_1,\mathfrak{g}_i]=\mathfrak{g}_{i+1},\quad
\mathfrak{g}_k\neq\{0\},\quad [\mathfrak{g}_1,\mathfrak{g}_k]=\{0\},
\end{equation}
where $[\mathfrak{g}_1,\mathfrak{g}_i]$ is the subspace of ${\mathfrak{g}}$ generated by
the commutators $[X,Y]$ with $X\in \mathfrak{g}_1$ and $Y\in \mathfrak{g}_i$. $\mathfrak{g}_1$ is called the \textit{first stratum} of the stratification and we will denote $m := \dim \mathfrak{g}_1 \leq n=\dim \mathfrak{g}$.
Choose a basis $e_1, \ldots , e_n$ of $\mathfrak{g}$ adapted to the stratification, that is such that
$$ e_{h_{j-1}+1}, \ldots ,e_{h_j}\ \mbox{ is a basis of}\  \mathfrak{g}_j\ \mbox{ for each }\ j = 1, \ldots ,k.$$ 
Let $X_1,\ldots ,X_n$ be the family of left invariant vector fields such that
at the identity element $e$ of $\mathbb{G}$ we have 
$X_i(e) = e_i$\, for every $i = 1, \ldots,n$.
Given \eqref{stratificazione}, the subset $X_1, \ldots , X_m$ generates by commutation all the other vector fields; we will refer to $X_1, \ldots , X_m$ as generating horizontal vector fields of the group. 
Given an element $x\in\G$, we denote by $\tau_x:\bbG\to\bbG$ the
{\em left translation} by $x$, which is given by
\[ \tau_x z \coloneqq x \cdot z\quad\mbox{for every }z\in\G,\]
where $\cdot$ is the group law in $\mathbb{G}$.
Moreover, it holds that the map $\tau_x$ is a smooth diffeomorphism, 
thus we can consider its
differential
${\rm d}_y\tau_x:T_y\G\to T_{x \cdot y}\G$ at any point $y\in\G$.

\subsection{Exponential map and Sub-Riemannian structures}
We recall that the exponential map $\exp : \mathfrak{g} \rightarrow \mathbb{G}$
is defined as
follows. First, we identify the Lie algebra $\mathfrak{g}$ with the tangent space at the identity $T_e \mathbb{G}$. Given any vector \(v\in T_e\mathbb{G}\) and denoting by \(\gamma\colon[0,1]\to\G\)
the (unique) smooth curve satisfying the ODE
\begin{equation}\label{eq:ODE_def_exp}
\left\{\begin{array}{ll}
\dot\gamma(t)={\rm d}_e\tau_{\gamma(t)}[v]\quad\text{ for every }t\in[0,1],\\
\gamma(0)=e,
\end{array}\right.
\end{equation}
we define \(\exp(v)=e^v\coloneqq\gamma(1)\), where ${\rm d}_e\tau_{\gamma(t)}[v]$ is a left-invariant vector field. It holds that \(\exp\) is
a diffeomorphism and any $p\in \G$ can be written in a unique way as 
\[
p=\exp(p_1X_1 + \dots + p_nX_n)=e^{p_1X_1 + \dots + p_nX_n},
\quad\text{ where }v = \sum_{i=1}^n p_i X_i.
\]
We can identify $p$ with the $n$-tuple $(p_1, \ldots, p_n)\in \mathbb{R}^n$ and $\G$ with $ \mathbb{R}^n$ where the group operation $\cdot$ satisfies (see \cite{Bon} and \cite[Section 7]{Monti})
\begin{equation*}
x\cdot y
= \exp\left( \exp^{-1}(x)\star \exp^{-1}(y)\right)\quad \mbox{for every}\ x,y\in \bbG,
\end{equation*}
where $\star$ denotes the group operation determined by the Campbell--Baker--Hausdorff formula, see e.g.\ \cite{Mag,BonLanUgu}.

The subbundle of the tangent bundle $T\G$ that is spanned by the vector fields $X_1,\ldots ,X_m$ plays a particularly important role in the theory, it is called the {\em horizontal bundle} $H\G$; the fibers of $H\G$ are
$$ H_x\G=\mathrm{span}\left\{X_1(x),\ldots, X_m(x)\right\}.$$
A sub-Riemannian structure can be defined on $\G$ in the following way. 
Consider a scalar product $\langle\cdot,\cdot\rangle_e$ on $\mathfrak{g}_1=H_e\G$
that makes
$\{ X_1, \ldots, X_m\}$ an orthonormal basis.
Moreover, by left translating the horizontal fiber in the identity, we obtain that
$H_x\G={\rm d}_e\tau_x(\mathfrak{g}_1)$ 
and, by \cite[Lemma 7.48]{agrac},
the map 
$$ T \mathbb{G} \ni (x, v) \mapsto (x, {\rm d}_x\tau_{x^{-1}}[v]) \in \mathbb{G} \times T_e \mathbb{G} $$
is an isomorphism between $T \mathbb{G}$ and $\mathbb{G} \times \mathfrak{g}$, in other words, the tangent bundle is trivial.
This allows us to define the scalar product $\langle\cdot,\cdot\rangle_x$ on $H_x\G$ as
\[
\langle v,w\rangle_x:=\big\langle{\rm d}_x\tau_{x^{-1}}[v],{\rm d}_x\tau_{x^{-1}}[w]
\big\rangle_e\quad\mbox{for every }v,w\in H_x\G.
\]
Notice that $\{X_1(x),\ldots,X_m(x)\}$ is an orthonormal basis of $H_x\G$
with respect to $\langle\cdot,\cdot\rangle_x$. We denote by $\norm{\cdot}_x$ the norm
induced by $\langle\cdot,\cdot\rangle_x$, namely $\norm{v}_x:=\sqrt{\langle v,v\rangle_x}$ for every $v\in H_x\G$. 
By the left invariance of the sub-Riemannian structure, for every $v \in H_x\mathbb{G}$ there exists
a unique vector $ \bar{v} \in H_e \mathbb{G}$ such that $v ={\rm d}_e\tau_x(\bar{v})$ and we get that
\begin{equation}\label{left}
\| \bar{v} \|_e = \| {\rm d}_e \tau_x [\bar{v}]\|_x = \| v \|_x \quad 
\text{for every } (x, v) \in H \mathbb{G}.
\end{equation} 
A further choice of the norm would not change the biLipschitz equivalence
class of the sub-Riemannian structure.
This is the reason why we may assume that the norm $\| \cdot\|_x$ is coming from a scalar product $\langle \cdot, \cdot \rangle_x$ (see \cite{ledonne}).
If $y = (y_1, \ldots, y_n) \in \mathbb{G}$ 
and $x \in \mathbb{G}$ are given, we set the \textit{projection map} as:
$$\pi_x : \mathbb{G}  \rightarrow H_x \mathbb{G} \;\;\; \mbox{as}\;\;\; \pi_x(y) = \sum_{j = 1}^m y_j X_j(x).$$
The map $y \mapsto \pi_x(y)$ is a smooth section of $H_x\mathbb{G}$ and it is linear in $y$. 
Furthermore, if $v \in \mathfrak{g}_1$, by exponential coordinates it holds that
\begin{equation}\label{proj}
\pi_x( e^v) = \pi_x (v_1, \ldots, v_m, 0 \ldots, 0) = \sum_{i = 1}^m v_i X_i(x) = {\rm d}_e \tau_x [v] 
\quad\text{for all } x \in \mathbb{G}\text{ and }v\in g_1.
\end{equation}
\subsection{Dilations and Carnot--Carath\'{e}odory distance} For any $\lambda >0$, we denote by $\delta^{\star}_\lambda:\mathfrak{g}\to \mathfrak{g}$ the unique linear map such that
\[ \delta^{\star}_\lambda{X}=\lambda^i X,\qquad \forall X\in \mathfrak{g}_i.\]
The maps $\delta^{\star}_{\lambda}:\mathfrak{g}\to\mathfrak{g}$ are Lie algebra automorphisms, i.e., $\delta^{\star}_{\lambda}([X, Y ]) = [\delta^{\star}_{\lambda}X, \delta^{\star}_{\lambda}Y ]$ for all $X,Y\in \mathfrak{g}$. For every $\lambda>0$, the map $\delta^{\star}_\lambda$ naturally induces an automorphism on the group $\delta_{\lambda}:\bbG\to \bbG$ by the identity 
\begin{equation}\label{dilinduced}
  \delta_{\lambda}(x)=(\exp \circ\, \delta^{\star}_{\lambda}\circ \log) (x),  
\end{equation} 
where with $\log$ we denote the inverse map of $\exp$. It is easy to verify that both the families $(\delta^*_{\lambda})_{\lambda>0}$ and $(\delta_\lambda)_{\lambda>0}$ are one-parameter groups of automorphisms (of Lie algebras and of groups, respectively), namely,
$\delta^{\star}_{\lambda}\circ \delta^{\star}_{\eta}= \delta^{\star}_{\lambda\eta}$ 
and  $\delta_{\lambda}\circ \delta_{\eta}= \delta_{\lambda\eta}$ for all $\lambda, \eta>0$.
The maps $\delta^{\star}_\lambda, \delta_\lambda$ are both called \emph{dilations of factor $\lambda$}. \\
Let us remark that,
since \(\delta_\lambda^\star v=\lambda v\) for every \(v\in\g_1=H_e\G\) and thanks to \eqref{dilinduced},
one can easily realize that
\begin{equation}\label{eq:exp_on_horizontal}
\delta_\lambda\exp(v)=\exp(\lambda v)\quad\text{ for every }
v\in H_e\G\text{ and }\lambda>0.
\end{equation}
According to \cite{Pauls}, we can extend dilations also to negative parameters $\lambda < 0$, denoting
$\delta^{\star}_{| \lambda|}(X) = \delta^{\star}_{\lambda}(- X) = |\lambda|^i (- X)$ for $X\in \mathfrak{g}_i$ and, in the present paper, we exploit this fact only on the fist layer $\mathfrak{g}_1$.
Indeed, it holds that $\pi_x(\delta_\lambda e^w) =\lambda \pi_x(e^w)$ for every $w \in \mathfrak{g}_1$ and $\lambda>0$. 
Since the dilations are defined only on the Lie algebra $\mathfrak{g}$, we extend them, by left translations, on the entire $T_x \mathbb{G}$, for every $x \in \mathbb{G}$. This allows us to write
$\delta^*_{\lambda}v = \lambda v$ for every  $\lambda > 0$, $x \in \mathbb{G}$ and $v \in H_x \mathbb{G}$.

\begin{definition}\label{horizontalcurve}
An absolutely continuous curve
$\gamma\colon [a,b]\to \bbG$ is said to be \emph{horizontal} if there exists a vector of measurable functions $h=(h_1(t),\ldots h_m(t)): [a,b]\to \mathbb{R}^{m}$ called the vector of canonical coordinates, such that
	\begin{itemize}
	\item $\dot{\gamma}(t)=\sum_{i=1}^m h_i(t)X_i(\gamma(t))$ for a.e.\ $t\in [a,b]$;
	\item $ \lvert h \rvert\in L^{\infty}(a,b)$.
	\end{itemize}
	The \emph{length} of such a curve is given by
	$L_{\bbG}(\gamma)=\int_a^b  \norm{\dot{\gamma}(t)}_{\gamma(t)} \,{\rm d}t.$
\end{definition}
Chow--Rashevskii's theorem \cite[Theorem 19.1.3]{BonLanUgu} asserts that any two points in a Carnot group can be connected by a horizontal curve. Hence, the following definition is well-posed.
\begin{definition}\label{carnotdistance}
	For every $x,y\in \mathbb{G}$, the \emph{Carnot--Carath\'{e}odory (CC) distance} is defined by
	\[ d_{cc}(x,y)=\inf \left\{L_{\bbG}(\gamma)\colon \gamma \text{ is a horizontal curve joining } x \text{ and }y\right\}.\]
\end{definition}

We remark that, by Chow--Rashevskii's Theorem, the Carnot--Carath\'eodory distance is finite. Moreover, it is homogeneous with respect to dilations and left translations, 
more precisely, for every $\lambda>0$ and for every $x,y, z\in \G$ one has
\[
d_{cc}(\delta_\lambda x,\delta_\lambda y)=\lambda d_{cc}(x,y),\qquad d_{cc}(\tau_x y,\tau_x z)=d_{cc}(y,z).
\]
This immediately implies that $\tau_x(B(y,r))=B(\tau_x y, r)$ and $\delta_\lambda B(y,r)=B(\delta_\lambda y, \lambda r)$, where  
\[ B(x,r)=\big\{y\in \bbG: d_{cc}(y, x) < r\big\}\]
is the open ball centered at $x\in \bbG$ with radius $r>0$. 
\medskip

In the sequel, we will need the following crucial estimate, proved in \cite[Theorem 1.5.1]{Monti}.
\begin{thm}\label{sea} Let $\mathbb{G}$ be a Carnot group of step $k$
and let $K\subset\mathbb{G}$ be a compact set. Then there exists $C_K=C(K) > 1$ such that 
\begin{equation}
C_K^{-1} \lvert x-y \rvert \leq d_{cc}(x,y)\leq C_K \lvert x-y \rvert^{\frac{1}{k}},\qquad \forall x,y\in K.
\end{equation}
\end{thm}

The following lemma shows the biLipschitz equivalence between the Carnot--Carath\'{e}odory distance and 
the norm induced by the scalar product.

\begin{lem}\label{lem:equiv_norm_dcc}
There exists a constant \(c\geq 1\) such that
\begin{equation}\label{eq:equiv_norm_dcc}
\frac{1}{c}\norm{v}_e \leq d_{cc}(e,\exp v)\leq c \norm{v}_e
\quad\text{ for every }v\in \mathfrak{g}_1.
\end{equation}
\end{lem}
\begin{proof}
Denote by \(S\) the unit sphere of \((H_e\mathbb G,\norm{\cdot}_e)\),
namely \(S:=\{v\in H_e\mathbb G\,:\,\norm{v}_e=1\}\). 
Define the function \(\eta: S\to[0,+\infty)\) as \(\eta(v)\coloneqq d_{cc}(e,\exp v)\)
for every \(v\in S\). By Theorem \ref{sea}, \(\eta\) is continuous on the compact set \(S\). Then we can
find \(c\geq 1\) such that \(1/c \leq \eta(v) \leq c\) holds for every
\(v\in S\). We can thus conclude by \(1\)-homogeneity: since
\(d_{cc}(e,\exp(\lambda v))=\lambda d_{cc}(e,\exp v)\)
for every \(\lambda>0\) and \(v\in S\), we deduce that
\(\eta(v/\|v\|_e)=d_{cc}(e,\exp v)/\|v\|_e\) for every
\(v\in H_e\mathbb G \setminus\{0\}\) and thus
\[
\frac{1}{c}\leq\frac{d_{cc}(e,\exp v)}{\|v\|_e}\leq c
\quad\text{ for every }v\in H_e\mathbb G\setminus\{0\},
\]
which yields \eqref{eq:equiv_norm_dcc}.
\end{proof}

\subsection{Differentiability in Carnot Groups}
We recall some basic definitions regarding differentiability in Carnot groups. 
\begin{defn} A map $L : \mathbb{G} \rightarrow \mathbb{R}$ is called a \emph{homogeneous homomorphism} if 
 \[ L(x \cdot y) = L(x) + L(y) \quad\text{and}\quad
  L(\delta_\lambda(x)) = \lambda L(x)  \quad\text{for every } x, y \in \mathbb{G} \quad\text{and}\quad \lambda>0.\]
\end{defn}

Now we are ready to introduce the following fundamental notion of  differentiability, see \cite{Pansu}.

\begin{defn}\label{diff} Let $\Omega\subset \bbG$ be an open subset. A map $f : \Omega \rightarrow \mathbb{R}$ is \emph{Pansu differentiable} at $x \in \Omega$ if there exists a homogeneous homomorphism 
$L_x : \mathbb{G} \rightarrow \mathbb{R}$ such that
\begin{equation*}
\lim_{y \rightarrow x} \frac{f(x) - f(y) - L_x [y^{-1} \cdot x]}
{d_{cc}(y, x)} = 0.
\end{equation*}
The map $L_x \coloneqq d_{\mathbb{G}} f(x):\G\to \mathbb{R}$ is called Pansu differential of $f$ at $x$.
\end{defn}
\begin{rem}\label{scalprod}
{\rm
If $f:\Omega\to \mathbb{R}$ is differentiable at $x \in\Omega$, then $X_j f(x)$ exists for any
$j = 1, \ldots, m$, and for any $v \in \mathbb{G}$ we have
\begin{equation*}
d_{\mathbb{G}}f(x)[v] = \langle \nabla_{\mathbb{G}} f(x), \pi_{x} (v)\rangle_{x},
\end{equation*}
where the horizontal gradient $\nabla_\G f(x)$ is defined as
$$\nabla_\G f(x) \coloneqq \sum_{i = 1}^m X_i f(x) X_i(x).$$
We stress that the notion of the horizontal gradient only depends on the choice of the
generating horizontal vector fields
and hence it is uniquely determined by the sub-Riemannian metric chosen 
on $\Omega \subset \mathbb{G}$.
\fr}\end{rem}
Unless otherwise specified, by a Lipschitz function $f:\Omega\to \bbR$ we mean
a function that is Lipschitz with respect to the Carnot--Carath\'{e}odory distance
$d_{cc}$, namely there exists a constant $C \geq 0$ such that $\lvert f(x)-f(y) \rvert\leq C d_{cc}(x,y)$
for every $x,y\in\Omega$. Moreover, $f : \mathbb{G} \rightarrow \mathbb{R}$ is said to be a locally Lipschitz function if it is Lipschitz on every open bounded set $\Omega \subset \mathbb{G}$.
The notion of Pansu differentiability is motivated by the following result due to Pansu \cite{Pansu} (see also \cite{MPS} for a similar result in a more general setting). In general, it states that any Lipschitz map between two Carnot groups has almost everywhere a differential which is a homogeneous homomorphism.
Below we report the statement of Pansu's result only for real-valued Lipschitz maps, since this is sufficient for our purposes.

\begin{thm}\label{Pansu}
Let $\Omega\subset \bbG$ be an open subset. Then for every Lipschitz function $f:\Omega\to \bbR$ we have that $f$ is Pansu differentiable at $\mathcal{L}^n$-a.e.\ $x \in \Omega$.
\end{thm}

Let $x \in \mathbb{G}$ and
$\bar{v} \in \mathfrak{g}_1$, then 
the map $ t \mapsto x \cdot \delta_t \exp(\bar{v})$ is Lipschitz. Hence,
if $f : \mathbb{G} \rightarrow \mathbb{R}$ is Lipschitz, then the composition $t \mapsto 
f(x \cdot \delta_t \exp(\bar{v}))$ is a Lipschitz mapping from $\mathbb{R}$ to itself, hence Pansu differentiable.
By Definition \ref{diff}, Remark \ref{scalprod}, and \cite[Theorem 4.6]{PS2}, it is easy to verify that
\begin{equation}\label{diff2}
\langle \nabla_{\mathbb{G}}f(x), \pi_x(v) \rangle =
\lim_{t \rightarrow 0}  \frac{f(x \cdot \delta_t \,e^{\bar{v}}) - f(x)}{t}
\quad 
\mbox{ for a.e.\ }x\in\G\mbox{ and for every } v  \in H_x \mathbb{G},
\end{equation}
where $\bar{v} = {\rm d}_x \tau_{x^{-1}} [v]$.

\subsection{Sub-Finsler Metrics and Duality}\label{s:subFins_duality}
Inspired by \cite{Vent}, now we introduce the following definition. 
Let $\mathbb{G}$  be a Carnot metric group and let $\Omega\subset \mathbb{G}$ be an open set. If $\alpha \geq 1$, 
we introduce the family $\mathcal{D}_{cc}(\Omega)$ containing all the geodesic distances $d: \Omega \times \Omega \rightarrow[0,+\infty)$ verifying
\begin{align}\label{ineq}
\frac{1}{\alpha} d_{cc}(x, y) \leq d(x, y) \leq \alpha d_{cc}(x, y)\qquad \forall x,y \in \Omega.
\end{align}
Therefore, $\mathcal D_{cc}(\Omega)$ depends on $\alpha$
and we omit such dependence for the sake of brevity. 
Notice that we may have $\mathcal{D}_{cc}(\Omega) = \emptyset$ if the domain $\Omega \subset \mathbb{G}$ is disconnected or it has an irregular boundary.
We endow $\mathcal{D}_{cc}(\Omega)$ with the topology of
the uniform convergence on compact subsets of $\Omega \times \Omega$.
We will see in the Proof of Theorem \ref{gammaconvergence} that 
$\mathcal{D}_{cc}(\Omega)$ is compact with respect to such topology.

\begin{defn} \label{defF} For $\alpha \geq 1$, we define $\mathcal{M}_{cc}^{\alpha}(\mathbb{G})$ as the family of
all those maps $\varphi : H \mathbb{G}\rightarrow [0, + \infty)$, that we will call
metrics on $H\G$, verifying the following properties:
\begin{itemize}
\item[(1)] \(\varphi:H\G\to\mathbb R\) is Borel measurable,
where \(H\G\) is endowed with the product $\sigma$-algebra;
\item[(2)] $\varphi(x, \delta_{\lambda}^* v) = \lvert \lambda \rvert \varphi(x, v)$ for every $(x,v)\in H\mathbb{G}$ and $\lambda \in \mathbb{R}$;
\item[(3)] $\frac{1}{\alpha} \|v\|_x \leq \varphi(x, v) \leq \alpha \|v\|_x$ for every $(x, v) \in H \mathbb{G}$.
\end{itemize}
Moreover, we will say that $\varphi \in \mathcal{M}_{cc}^\alpha(\mathbb{G})$ is a \emph{sub-Finsler convex metric} if
\begin{align} 
\varphi(x,v_1 + v_2)\leq \varphi(x,v_1)+\varphi(x,v_2)
\end{align}
for every $x\in\mathbb{G}$ and $v_1, v_2\in H_x \mathbb{G}$
(or equivalently if $\varphi(x,\cdot)$ is a norm for every $x\in\G$).
\end{defn}
According to the preliminaries,
conditions $(2)$ and $(3)$ are well-defined with respect to the exponential and
the dilation map.
Moreover, let us remark that condition $(1)$ is equivalent to the Borel measurability with respect to 
the product space $\mathbb{G} \times \mathfrak{g}$.
\medskip

Our next aim is to introduce the dual of a metric belonging to $\mathcal M_{cc}^{\alpha}(\mathbb{G})$.

\begin{defn}[Dual Metric] Let us take $\varphi \in \mathcal{M}_{cc}^{\alpha}(\mathbb{G})$. We define 
the \emph{dual metric} $\varphi^{\star} : H \mathbb{G} \rightarrow[0,+\infty)$ of  $\varphi$ as
\begin{equation}\label{dual}
\varphi^{\star}(x, v) \coloneqq
\sup \Biggl\{    
\frac{\lvert \langle v, w \rangle_x \rvert}
{\varphi (x, w)} \; : \; w \in H_x\mathbb{G}, \; w \neq 0
\Biggr\}.
\end{equation}
\end{defn}

In general, the dual metric enjoys many useful properties, as we can see below.
\begin{prop}\label{properties} For any\, $\varphi \in \mathcal{M}_{cc}^{\alpha}(\mathbb{G})$, 
it holds that $\varphi^{\star}$ is a sub-Finsler convex metric,
and in particular
\begin{equation}\label{ineq_dual}
\frac{1}{\alpha} \|v \|_x \leq \varphi^{\star}(x, v) \leq \alpha \| v \|_x\quad\mbox{for every }(x,v)\in H\G.
\end{equation}
\end{prop}
\proof 
It is straightforward to prove property $(2)$
since for every $v, w \in H_x \mathbb{G}$ and $\lambda \in \mathbb{R}$ we have that
$
\langle \delta_{\lambda}^* v, w\rangle_x = \lambda \langle 
v, w \rangle_x.
$
Passing to the supremum over all $w \in H_x \mathbb{G} \setminus \{0\}$, we obtain that $\varphi^{\star}(x, \delta_{\lambda}^* v) = \lvert \lambda \rvert\varphi^{\star}(x,v) $.
In order to prove the convexity on the horizontal bundle, let us consider $v_1, v_2 \in H_x \mathbb{G}$, then we obtain:
\begin{align*}
\varphi^{\star}(x, v_1 + v_2) 
&\leq \sup \Biggl\{ 
\frac{\lvert \langle v_1, w \rangle_x \rvert} 
{\varphi (x, w)} + \frac{ \lvert \langle v_2, w \rangle_x \rvert} 
{\varphi (x, w)}
\; :\; w \in H_x \mathbb{G}, \; w \neq 0\Biggr\} \\
&\leq  \varphi^{\star}(x, v_1) + \varphi^{\star}(x, v_2).
\end{align*}
Moreover, if we take  $w \in H_x\mathbb{G}\setminus\{0\}$
it holds that 
\[ 
\frac{1}{\alpha} \frac{\lvert \langle v, w \rangle_x \rvert}{\| w \|_x}
\leq
\frac{\lvert \langle v, w \rangle_x \lvert}{\varphi(x, w)} \leq
\alpha \frac{\lvert \langle v, w \rangle_x \rvert}{\| w \|_x}.
\]
By taking the supremum over all $w \in H_x \mathbb{G}\setminus\{0\}$,
we obtain \eqref{ineq_dual} and accordingly
property $(3)$ of Definition \ref{defF}.
Therefore, $\varphi^{\star}(x, \cdot)$ is a norm, thus in particular it is
continuous. Finally, chosen a dense sequence \((w_n)_n\) in $\mathfrak{g}_1\setminus\{0\}$,
we have that $({\rm d}_e\tau_x[w_n])_n$ is dense in $H_x\G$ for every $x\in\G$,
thus for any $v\in\mathfrak g_1$ we can write
\[
\varphi^\star(x,{\rm d}_e\tau_x[v])=\sup_{n\in\mathbb N}
\frac{\lvert\langle{\rm d}_e\tau_x[v],{\rm d}_e\tau_x[w_n]\rangle_x\rvert}
{\varphi(x,{\rm d}_e\tau_x[w_n])}=\sup_{n\in\mathbb N}
\frac{\lvert\langle v,w_n\rangle_e\rvert}{\varphi(x,{\rm d}_e\tau_x[w_n])}
\]
for every $x\in\G$, which shows that $\G\ni x\mapsto\varphi^\star(x,{\rm d}_e\tau_x[v])$
is measurable and accordingly property $(1)$ of Definition \ref{defF} is satisfied.
All in all, $\varphi^\star$ is a sub-Finsler convex metric.
\endproof

We can characterize sub-Finsler convex metrics $\varphi$ in terms of the bidual metric $\varphi^{\star \star}$.
\begin{prop}\label{bidual} Let us consider $\varphi \in \mathcal{M}_{cc}^\alpha(\mathbb{G})$.
Then $\varphi$ is a sub-Finsler convex metric if and only if   
$\varphi(x, v)= \varphi^{\star \star}(x, v)$ for every $(x, v) \in H \mathbb{G}$.
\end{prop}
\begin{proof}
\ \\
\boxed{\Leftarrow} Since $\varphi$ is $1$-homogeneous in the second entry, in order to prove that $\varphi(x,\cdot)$ is convex on 
$H_x \mathbb{G}$, it is sufficient to prove that $\varphi(x, v_1+v_2)\leq \varphi (x, v_1)+\varphi(x, v_2)$ for every $x\in\G$ and $v_1,v_2\in H_x \mathbb{G}$.
By assumption $\varphi(x, v)= \varphi^{\star\star}(x, v)$ for all $(x, v) \in H \mathbb{G}$, then for every $v_1, v_2 \in  H_x \mathbb{G}$  we can write
\begin{align*}
\varphi(x, v_1 + v_2) &= \varphi^{\star\star}(x, v_1 + v_2)\\
&\leq \sup \Biggl\{ 
\frac{\lvert \langle v_1, w \rangle_x \rvert} 
{\varphi^{\star}(x, w)} + \frac{\lvert \langle v_2, w \rangle_x \lvert} 
{\varphi^{\star}(x, w)}
\; :\; w \in H_x \mathbb{G}, \; w \neq 0\Biggr\} \\
&\leq \varphi^{\star\star}(x, v_1) + \varphi^{\star \star}(x, v_2)\\
&=\varphi(x, v_1) + \varphi(x, v_2).
\end{align*}
\boxed{\Rightarrow} Given that $\varphi(x,\cdot)$ is convex and $1$-homogeneous, $\varphi(x,\cdot)$ is a norm on $H_x \mathbb{G}$ and $\varphi^{\star \star}(x,\cdot)$ represents its bidual norm. Since each horizontal fiber is finite-dimensional, the conclusion follows.
\end{proof}

At the end, we present the following properties that we will need in the last section.
\begin{lem}\label{semic}
Let \(\varphi\in\mathcal M_{cc}^\alpha(\mathbb G)\) be a sub-Finsler convex metric.
Then the following hold:
\begin{itemize}
\item[\(\rm i)\)] If \(\varphi\) is lower semicontinuous, then
\(\varphi^\star\) is upper semicontinuous.
\item[\(\rm ii)\)] If \(\varphi\) is upper semicontinuous, then
\(\varphi^\star\) is lower semicontinuous.
\end{itemize}
In particular, if \(\varphi\) is continuous, then \(\varphi^\star\)
is continuous.
\end{lem}
\begin{proof}
To prove i) suppose \(\varphi\) is lower semicontinuous. Fix
$(x,v) \in H \mathbb{G}$ and $(x_n,v_n)\in H\mathbb G\) such that \((x_n,v_n)\to(x,v)\),
in the sense that $ d_{cc}(x_n, x) +\big\|{\rm d}_{x_n}\tau_{x_n^{-1}}[v_n] - {\rm d}_x\tau_{x^{-1}}[v] \big\|_e \to 0$.
Possibly passing to a not relabeled subsequence, we can assume
that \(\limsup_n\varphi^\star(x_n,v_n)\) is actually a limit.
Given any \(n\in\mathbb{N}\), there exists \(w_n\in H_{x_n}\mathbb G\) such that
\(\varphi(x_n,w_n)=1\) and \(\varphi^{\star}(x_n,v_n)=\lvert \langle v_n,w_n\rangle_{x_n} \rvert \).
Since the fiber of the horizontal bundle
is compact, then
there exists \(w\in H_x\mathbb G\) such that (up to a not
relabeled subsequence) \((x_n,w_n)\to(x,w)\). Being \(\varphi\) lower semicontinuous,
we deduce that
$$\varphi(x, w) \leq \liminf_{n \rightarrow \infty} \varphi(x_n, w_n) \leq 1.$$
Therefore, we conclude that
\[
\varphi^\star(x,v)\geq\frac{\lvert \langle v,w\rangle_x \rvert}{\varphi(x,w)}
\geq\lim_{n\to\infty} \lvert \langle v_n,w_n\rangle_{x_n} \rvert=
\limsup_{n\to\infty}\varphi^\star(x_n,v_n),
\]
which proves that \(\varphi^\star\) is upper semicontinuous.\\
The assertion ii) can be proved
noticing that if \(\varphi\) is upper semicontinuous, then \(\varphi^\star\) is lower semicontinuous
as it can be expressed as a supremum of lower semicontinuous functions.
\end{proof}
\begin{notat}{\rm
For any $d \in \mathcal{D}_{cc}(\Omega)$ and $a\in \Omega$, we denote by
$d_a : \Omega \rightarrow [0,+\infty)$ the fixed-point distance map $d_a(x) \coloneqq d(a, x)$.
Clearly, $d_a$ is a Lipschitz function
and, by Theorem \ref{Pansu}, $d_a$  is Pansu differentiable for a.e.\ $x\in \Omega$. 
We denote with $\Lip([0,1], \Omega)$ the class of all Lipschitz continuous curves $\gamma : [0,1] \rightarrow \Omega$
and with $\mathcal{H}([0,1], \Omega)$ the set of horizontal curves. In the sequel, sometimes we omit the unit interval since we are going to consider curves defined on and we assume that such curves are parametrized with constant velocity. 
Moreover, for every Lebesgue null set $N\subset \Omega$, we set $\mathcal{P}(\Omega, N)$
as the class of all horizontal curves $ \gamma: [0,1]\to  \Omega$ such that 
\[\mathcal{L}^1(\left\{t \in [0,1]\; | \;\gamma(t) \in N\right\}) = 0,\] where $\mathcal{L}^1$ is the standard $1$-dimensional Lebesgue measure. By \cite[Lemma 2.2]{DCP4} we have that $\mathcal{P}(\Omega, N) \neq \emptyset$. 
Furthermore, with $H\Omega$ we mean the restriction of the horizontal bundle $H\G$
to $\Omega$, i.e.,
\[
H\Omega:=\{(x,v)\in H\G:x\in\Omega\}.
\]
If not otherwise stated, for every $v \in T_x \G$ and $x \in \G$ sometimes we will denote by $\bar{v} \coloneqq {\rm d}_x \tau_{x^{-1}}[v]$ the representative vector of $T_x \mathbb{G} \ni v$ in the Lie algebra $\mathfrak{g}$.
\fr}\end{notat}

\section{Metric Derivative and Length Representation Theorem}
Given a geodesic distance $d \in \mathcal{D}_{cc}(\G)$, it is quite natural to consider the associated metric given by differentiation. 
The next definition is inspired by the ones proposed in \cite{Pauls,Vent} but, in our setting, we necessarily have to 
define it only on the horizontal bundle $H \G$. 
\begin{defn}[Metric derivative] Given any $d \in \mathcal{D}_{cc}(\G)$, we define the map $\varphi_d : H \mathbb{G} \rightarrow [0, +\infty)$ as
$$\varphi_d(x, v) \coloneqq \limsup_{t \rightarrow 0}
\frac{d(x, x \cdot \delta_t \exp{{\rm d}_x\tau_{x^{-1}}[v])}}{\lvert t \rvert}\quad\mbox{for every }(x,v)\in H\G.$$
\end{defn}
Note that we translate the vector $v \in H_x \mathbb{G}$ to $e$ via the differential of the left traslation,
because the exponential map is defined on 
the first stratum $\mathfrak{g}_1=H_e\G$. The next Lemma tells us that the metric derivative is actually a metric.
\begin{lem}
For every $d \in \mathcal{D}_{cc}(\mathbb{G})$ we have that 
$\varphi_d\in \mathcal{M}_{cc}^{c \alpha}(\mathbb{G})$, for some $c \geq 1$
independent of $d$.
\end{lem} 
\begin{proof}
In order to prove $(1)$, let just observe that
\[
\varphi_d(x,v)=\lim_{n\to\infty}\sup_{\substack{t\in\mathbb Q:\\ \lvert t \rvert<1/n}}
\frac{d(x,x\cdot\delta_t \exp{{\rm d}_x\tau_{x^{-1}}}[v])}{\lvert t \rvert}\quad\mbox{for every }(x,v)\in H\G.
\]
Let us verify $(2)$. Pick $x\in \mathbb{G}$, $v\in H_x \mathbb{G}$ and $t,\lambda\in\bbR$.
Since the differential of the left translation is a diffeomorphism and from the equality \eqref{eq:exp_on_horizontal}
we have that $ \delta_t\exp\big({\rm d}_x \tau_{x^{-1}}
[\delta_{\lambda}^\star(v)]\big)
= \delta_t \delta_{\lambda}\exp\big({\rm d}_x \tau_{x^{-1}}[v]\big)$. Therefore
\begin{align*}
\varphi_d(x, \delta_{\lambda}^* v)=
\limsup_{t\to 0}\frac{d(x,x \cdot \delta_t\delta_\lambda e^{{\rm d}_x \tau_{x^{-1}}[v]})}
{ \lvert t \rvert}=\lvert \lambda \rvert \limsup_{t\to 0}\frac{d(x,x \cdot \delta_{t \lambda} 
e^{{\rm d}_x \tau_{x^{-1}}[v])})}
{ \rvert t\lambda \lvert}=\lvert \lambda \rvert \varphi(x,v).
\end{align*}
In order to show $(3)$, fix $x\in \mathbb{G}$ and $v\in H_x \mathbb{G}$. Since $d\in \mathcal{D}_{cc}(\mathbb{G})$ we can write 
\begin{align*}
\varphi_d(x,v)&\leq \alpha \limsup_{t\to 0}\frac{d_{cc}(x,x \cdot \delta_t e^{{\rm d}_x \tau_{x^{-1}}[v]})}
{\lvert t \rvert}
=\alpha \, d_{cc}(e, \exp {\rm d}_x \tau_{x^{-1}}[v])
\leq c \, \alpha \| {\rm d}_x \tau_{x^{-1}}[v]\|_e
\end{align*}
where in the last inequality we applied Lemma \ref{lem:equiv_norm_dcc}. The estimate from below can be proved similarly.
Finally, using the left invariance of the norm, for every $(x, v) \in H \mathbb{G}$, we get that
\[\frac{1}{c \alpha} \| v\|_x \leq \varphi_d(x, v) \leq c \,\alpha \| v\|_x \]
and the conclusion follows.
\end{proof}

The next result comes from \cite[Proposition 1.3.3]{Monti} and a general proof can be find in \cite[Proposition 3.50]{agrac}. It says that Lipschitz curves and horizontal ones essentially coincide when the $L^{\infty}$-norm of the canonical coordinates is finite.
\begin{prop}\label{LipHor}
 A curve $\gamma:[a, b]\to \Omega \subset \G$ is Lipschitz if and only if it is horizontal and $\norm{h}_{L^\infty(a,b)}\leq L$, where $L$ is the Lipschitz constant.
\end{prop}

In general, if $(M,d)$ is a metric space and $\gamma:[0,1]\to M$ a Lipschitz curve, then
the \emph{classical metric derivative} is defined as
$$\lvert \dot{\gamma}(t) \rvert_d \coloneqq \lim_{s \rightarrow 0} \frac{d(\gamma(t + s),\gamma(t))}{ \lvert s \rvert}.$$
The existence of the previous limit is a general fact that holds in any metric space (see \cite[Theorem 2.7.6]{Burago}); indeed $\lvert \dot{\gamma}(t) \rvert_d $ exists for a.e.\ $t\in[0,1]$, it is a measurable function and it satisfies the equality 
\begin{equation}\label{dmetric}
L_d(\gamma)=\int_0^1 \lvert \dot{\gamma}(t) \rvert_d\,{\rm d}t .
\end{equation}
 where the \emph{classical length functional} of a rectifiable curve is defined as
 $$L_d(\gamma) = \sup_{ \{0 \leq t_1 < \ldots < t_k \leq 1\}}
\sum_{i = 1}^{k - 1} d(\gamma(t_{i + 1}), \gamma(t_i)),$$ and the supremum is taken over all finite partitions of $[0, 1]$.\\

Carnot groups are naturally endowed with sub-Riemannian distances which make them interesting examples of metric spaces $(\mathbb{G},d_{cc})$. In particular, the metric derivative can be explicitly computed 
(see \cite[Theorem 1.3.5]{Monti}).
\begin{lem}\label{le1}
Let $\gamma : [0,1] \rightarrow \mathbb{G}$ be a Lipschitz curve and let $h \in L^{\infty}(0,1)^{m}$ be its vector of canonical coordinates. Then
\begin{align*}
\lvert &\dot{\gamma}(t) \rvert_{d_{cc}} =\lim_{s\to 0}\frac{d_{cc}(\gamma(t+s),\gamma(t))}{\lvert s \rvert}= \lvert h(t) \rvert \quad \text{for a.e.}\; t\in [0,1]\\
&\text{and} \quad \lim_{s \rightarrow 0} \delta_{\frac{1}{s}}\left( \gamma(t)^{-1}\cdot \gamma(t + s)\right) = \left(h_1(t), \ldots, h_m(t), 0, \ldots, 0\right)
\quad \text{for a.e. }\; t \in [0,1].
\end{align*}
\end{lem}

The second claim is proved in \cite[Lemma 2.1.4]{Monti} and it gives a characterization of horizontal curves in terms of canonical coordinates. 
Therefore, by Proposition \ref{LipHor}, a Lipschitz curve is horizontal and,
with abuse of notation, we set the following quantity: 
\[\exp{\dot{\gamma}(t)} \coloneqq \exp {\rm d}_{\gamma(t)} \tau_{{\gamma(t)}^{-1}}[\dot{\gamma}(t)] 
		= \big( h_1(t), \ldots, h_m(t), 0 \ldots, 0\big), 
		\quad \mbox{ for a.e. } t \in [0,1].\]
Now, let $d \in \mathcal{D}_{cc}(\mathbb{G})$, then $(\G, d)$ inherits the structure of a metric space and hence, for any horizontal curve $\gamma : [0,1] \rightarrow \mathbb{G}$ we get that
\begin{equation}\label{uguaglianza2} 
\lvert \dot{\gamma}(t) \rvert_d = \varphi_d(\gamma(t), \dot{\gamma}(t))
\quad\mbox{for a.e.\ }t\in[0,1].
\end{equation}

Finally, the identities \eqref{dmetric} and \eqref{uguaglianza2} imply the following well known length reconstruction result.
\begin{thm}\label{lunghezza} Let $d \in \mathcal{D}_{cc}(\mathbb{G})$. Then, for every horizontal curve $\gamma : [0,1] \rightarrow \mathbb{G}$ we have
\begin{equation}
\label{eq2}
L_d(\gamma) = \int_0^1 \varphi_d(\gamma(t), \dot{\gamma}(t)) \,{\rm d}t.
\end{equation}
Moreover, for every $x, y \in \mathbb{G}$ we have
$$d(x, y) = \inf \Biggl\{ \int_0^1 \varphi_d(\gamma(t), \dot{\gamma}(t)) \,{\rm d}t \;:\;
\gamma \in \mathcal{H}([0,1], \mathbb{G}), \;\, \gamma(0) = x, \; \gamma(1)=y \Biggr\}.$$
\end{thm}

\subsection{Convexity of \texorpdfstring{$\varphi_d$}{phi-d}}

The aim of the present section is to prove that if $d \in \mathcal{D}_{cc}(\G)$, then $\varphi_d$ is also a sub-Finsler convex metric.
In order to make this, first we have to recall some technical results.

\begin{lem}\label{FubiniLeb}
Let $ \psi \colon \mathbb{G}\to\mathbb{R}$ be a locally bounded, Borel function
and $v\in H_x\mathbb{G} \setminus\{0\}$. Then
\begin{equation}\label{eq:pt_Leb_claim}
\psi(x)=\lim_{t\searrow 0}\frac{1}{t}\int_0^t\psi(x\cdot \delta_s e^{\bar{v}})\,{\rm d}s,
\quad\text{ for }\mathcal L^n\text{-a.e.\ }x\in\mathbb{G}.
\end{equation}
\end{lem}
\begin{proof}
Given any fixed \(y\in\mathbb{G}\), we have that \(\bbR\ni t\mapsto\psi(y\cdot \delta_t e^{\bar{v}})\in\bbR\)
is a locally bounded and Borel function, thus an application of Lebesgue's differentiation
theorem guarantees that for \(\mathcal L^1\)-a.e.\ \(r\in\bbR\)
\begin{equation}\label{eq:pt_Leb_aux}
\psi(y \cdot \delta_r e^{\bar{v}})=\lim_{t\searrow 0}\frac{1}{t}\int_0^t
\psi(y\cdot \delta_{r+s} e^{\bar{v}})\,\,{\rm d}s.
\end{equation}
In particular, an application of Fubini's theorem ensures that the set
\[
\Gamma\coloneqq\big\{(y,r)\in\bbG \times\bbR\;\big|
\;\text{\eqref{eq:pt_Leb_aux} holds}\big\}
\]
has \(\mathcal L^{n+1}\)-full measure, thus for
\(\mathcal L^1\)-a.e.\ \(r\in\bbR\) we have that \eqref{eq:pt_Leb_aux}
holds for \(\mathcal L^n\)-a.e.\ \(y\in\mathbb{G}\). Fix any such \(r\in\bbR\)
and a \(\mathcal L^n\)-negligible set \(N\subset\mathbb{G}\) satisfying
\eqref{eq:pt_Leb_aux} for every \(y\in\mathbb{G}\setminus N\).\\
Calling \(\sigma_z : \bbG\to\bbG\) the right-translation map
\(\sigma_z w\coloneqq w\cdot z\) for every \(z,w\in\bbG\) and
defining \(N'\coloneqq\sigma_{\delta_r e^{\bar{v}}}(N)\), we thus have that
\(\psi(x)=\lim_{t\searrow 0}\frac{1}{t}\int_0^t\psi(x\cdot\delta_s e^{\bar{v}})
\,{\rm d} s\) holds for every \(x\in\bbG\setminus N'\). Here, we also
used the fact that $\delta_{r+s}e^{\bar{v}}=\delta_r e^{\bar{v}}\cdot\delta_s e^{\bar{v}}$, which
is in turn guaranteed by the fact that $\bar{v}$ belongs to the first layer (see \cite[Lemma 2.2]{Pauls}).
Therefore, in order to prove \eqref{eq:pt_Leb_claim} it is only left to check that
\(N'\) is \(\mathcal L^n\)-negligible. This can be achieved by exploiting
the right-invariance of the measure \(\mathcal L^n\)
(see e.g.\ \cite[Proposition 1.7.7]{Monti}), namely the fact that
\(\mathcal L^n(E\cdot z)=\mathcal L^n(E)\) holds whenever \(E\subset\mathbb G\)
is a Borel set and \(z\in\mathbb G\). Indeed, this implies that
\((\sigma_{\delta_r e^{\bar{v}}})_\#\mathcal L^n=\mathcal L^n\), because
for any Borel set \(E\subset\mathbb G\) it holds that
\[
(\sigma_{\delta_r e^{\bar{v}}})_\#\mathcal L^n(E)=
\mathcal L^n(\sigma_{\delta_r e^{\bar{v}}}^{-1}(E))=
\mathcal L^n(\sigma_{\delta_{-r}e^{\bar{v}}}(E))=
\mathcal L^n(E\cdot\delta_{-r}e^{\bar{v}})=\mathcal L^n(E).
\]
In particular, we conclude that
\(\mathcal L^n(N')=(\sigma_{\delta_r e^{\bar{v}}})_\#\mathcal L^n(N')
=\mathcal L^n(N)=0\), as required.
\end{proof}

\begin{lem}\label{firstlemma}
Let $d \in \mathcal{D}_{cc}(\Omega)$, $\varphi \in \mathcal{M}_{cc}^\alpha(\Omega)$, and $N \subset \Omega$ be such that
$\lvert N \rvert = 0$. Suppose that for every $\gamma \in \mathcal{P}(\Omega,N)$ we have that
\[ d(\gamma(0), \gamma(1))\leq \int_0^1 \varphi(\gamma(t), \dot{\gamma}(t))\,{\rm d}t.
\]
Then for every fixed $a \in \Omega$, for a.e.\ $x\in\Omega$, and for every $v \in H_x\G$, we have that
\[ \left| \langle \nabla_{\mathbb{G}}d_a(x), v \rangle_x \right| 
\leq  \liminf_{t \rightarrow 0}\frac{d(x, x \cdot \delta_t e^{\bar{v}})}{t} 
\leq  \limsup_{t \rightarrow 0}\frac{d(x, x \cdot \delta_t e^{\bar{v}})}{t} \leq \varphi(x, v).\]
\end{lem}

\begin{proof} Let $N$ be as in the hypothesis and $v \in H_x \mathbb{G}$. 
For $a \in \Omega$, let $E(a, v)$ be the set of all $x \in \Omega$ for which 
$d_a$ is Pansu differentiable for a.e.\ $x \in \Omega$ and the map 
$ [0,1] \ni t \mapsto x \cdot\delta_t e^{\bar{v}}$ belongs to 
$\mathcal{P}(\Omega, N)$, with $t$ small enough. Moreover, thanks to Lemma \ref{FubiniLeb}, we can assume that
$$ \lim_{t\searrow 0} \frac{1}{t} \int_0^t
\varphi( x \cdot \delta_s e^{\bar{v}}, v) \,{\rm d}s = \varphi(x, v).$$ 
By Pansu--Rademacher Theorem $ \lvert \Omega \setminus E(a,v) \rvert = 0$ and, if $x \in E(a, v)$, applying the identity \eqref{diff2} we have that 
$$\lim_{t \rightarrow 0} \frac{d_a(x) - d_a( x \cdot \delta_t e^{\bar{v}}) - 
\langle \nabla_{\mathbb{G}} d_a(x),\pi_x(\delta_t e^{- \bar{v}}) \rangle_x}{ \lvert t \rvert} = 0.$$
Hence, by the reverse triangle inequality we can assert that
\begin{align*}
 \left| \langle \nabla_{\mathbb{G}} d_a(x), v \rangle_x \right|
&\leq \left| \liminf_{t \rightarrow 0} \frac{d_a(x \cdot \delta_t e^{\bar{v}}) - d_a(x)}{t}\right|
\leq  \liminf_{t \rightarrow 0}\frac{d(x, x \cdot \delta_t e^{\bar{v}})}{t} \\
&\leq  \limsup_{t \rightarrow 0}\frac{d(x, x \cdot \delta_t e^{\bar{v}})}{t}
\leq \lim_{t\searrow 0} \frac{1}{t} \int_0^t
\varphi( x \cdot \delta_s e^{\bar{v}}, v) \,{\rm d}s = \varphi(x, v).
\end{align*}
Pick a countable dense subset $F \subset H_x \G$ and put $E(a) = \cap_{y \in F} E(a,y)$. Then
$\lvert \Omega \setminus E(a) \rvert = 0$ and for all $x \in E(a)$ and all $v \in H_x\G$
we obtain the same estimate above.
\end{proof}
We observe that the previous Lemma could be proved under more general conditions. 
Indeed, the distance needs only to be geodesic in its domain.

\begin{lem}\label{estimate} Let  $\varphi \in \mathcal{M}_{cc}^\alpha(\Omega)$
 be a sub-Finsler convex metric, let $d \in \mathcal{D}_{cc}(\Omega)$ and $\Theta \subset \Omega$ be a countable dense set of $\Omega$. If
$$ \norm{\varphi(x, \nabla_{\mathbb{G}} d_a(x))}_{\infty} \leq 1 \;\;\; \forall a \in \Theta,$$
then there exists $N \subset \Omega$ such that $\lvert N \rvert = 0$ and for every $\gamma \in \mathcal{P}(\Omega, N)$
$$d(\gamma(0), \gamma(1)) \leq \int_0^1 \varphi^\star(\gamma(t), \dot{\gamma}(t)) \,{\rm d}t.$$
\end{lem}
\begin{proof}Let $E$ be the subset of $x \in \Omega$ where the function $y \mapsto d_a(y)$ is Pansu-differentiable 
for every $a \in \Theta$.
Since $\Theta$ is countable, we have that $N = \Omega \setminus E$ has zero Lebesgue measure.
Let $\gamma \in \mathcal{P}(\Omega, N)$, pick $a \in \Theta$ and set $f(t) \coloneqq d_a(\gamma(t))$ for every $t \in [0,1]$, then
\begin{align*}
d_a( \gamma(1)) - d_a(\gamma(0)) &= f(1) - f(0)
\leq \int_0^1 \left| \frac{\rm d}{\dt} f(t)\right| \,{\rm d}t
 = \int_0^1 \big| \langle \nabla_{\mathbb{G}}d_a(\gamma(t)), h(t)\rangle_{\gamma(t)} \big| \,{\rm d}t \\
&\leq \int_0^1 \varphi\big(\gamma(t), \nabla_{\mathbb{G}}d_a(\gamma(t))\big) \varphi^\star
(\gamma(t), \dot{\gamma}(t))\,{\rm d}t
\leq \int_0^1 \varphi^{\star}(\gamma(t), \dot{\gamma}(t)) \,{\rm d}t.
\end{align*}
Now, by density of $\Theta$ in $\Omega$ we can choose a sequence $\{a_k\}_{k\in\mathbb{N}} \subset \Theta$ converging to $\gamma(0)$, obtaining
\begin{align*}
d(\gamma(0), \gamma(1)) = \lim_{k \rightarrow \infty} \big(d_{a_k}(\gamma(1)) -
d_{a_k}(\gamma(0))\big)
\leq \int_0^1 \varphi^\star(\gamma(t), \dot{\gamma}(t)) \,{\rm d}t.
\end{align*}
\end{proof}

\begin{thm} Let $d \in \mathcal{D}_{cc}(\Omega)$. Then 
$\varphi_d$ is a sub-Finsler convex metric. In particular, for almost all $x \in \Omega$ and all $v \in H_x\mathbb{G}$
\begin{equation}
\varphi_d(x, v) = \lim_{t \rightarrow 0} \frac{d(x, x \cdot \delta_t e^{\bar{v}})}{\lvert t \rvert}.
\end{equation}
\end{thm}

\begin{proof} Take a countable dense subset $\Theta$ of $\Omega$ and, for each $a \in \Theta$, 
we consider $\Sigma_a$ a negligible Borel subset of $\Omega$ 
which contains all points where $d_a$ is not Pansu-differentiable.
For every $(x, v) \in H\Omega$\, we define
$$\xi(x, v) \coloneqq 
\begin{cases}
\sup_{a \in \Theta} \big| \langle \nabla_{\mathbb{G}}d_a(x), v \rangle_x \big|  \quad\text{ if} \,\; x \in \Omega \setminus \bigcup_{a\in\Theta}\Sigma_a;
\\
0 \quad\mbox{ otherwise}. 
\end{cases}
$$
For $\varepsilon > 0$ we define $\xi_{\varepsilon} : H\Omega \rightarrow [0, + \infty)$ as 
\,$\xi_{\varepsilon}(x, v) \coloneqq \xi(x, v) + \varepsilon\| v\|_x$, that is a Borel measurable function in $H \Omega$
and it is a sub-Finsler convex metric. Indeed, if we take $v_1, v_2 \in H_x \mathbb{G}$ we can estimate in this way
\begin{align*}
\xi_{\varepsilon}(x, v_1 + v_2) &= \sup_{a \in \Theta} \big| \langle \nabla_{\mathbb{G}}d_a(x), v_1 + v_2 \rangle_x \big| 
+ \varepsilon \| v_1 + v_2\|_x  \\
& \leq \xi(x, v_1) + \xi(x, v_2) + \varepsilon \| v_1 + v_2\|_x \leq  \xi_{\varepsilon}(x, v_1) + \xi_{\varepsilon}(x, v_2).
\end{align*}
The homogeneity w.r.t.\ the second variable comes from the equality
$  {\rm d}_e \tau_x [ \delta_\lambda^\star \bar{v}] = \lambda  {\rm d}_e \tau_x[\bar{v}]$
where $ \bar{v} = {\rm d}_x \tau_{x^{-1}}[v]$.
Moreover, if $a \in \Theta$ we get that
$$ \big| \langle \nabla_{\mathbb{G}}d_a(x), v \rangle_x \big| \leq \xi(x, v)
\leq \xi_{\varepsilon}(x, v) \quad\text{for a.e.\ } x \in \Omega \,\;\text{and } v \in H_x \mathbb{G}.$$
Thus, by definition of dual metric, we have
\begin{equation}\label{du}
\frac{\big| \langle \nabla_{\mathbb{G}}d_a(x), v \rangle_x \big|}{\xi_{\varepsilon}(x,v)} \leq 1  \;\; \Rightarrow \;\; 
\big\| \xi_{\varepsilon}^{\star}(x, \nabla_{\mathbb{G}} d_a(x))\big\|_{\infty} \leq 1.
\end{equation}
Being $\Theta$ countable, by Lemma \ref{estimate}, there exists a Lebesgue null set $N \subset \Omega$ such that,
the horizontal curve $\gamma(t) = x \cdot \delta_t e^{\bar{v}}$ belongs to $\mathcal{P}(\Omega, N)$,
and for every small $t >0$ we can infer that 
$$d(x, x \cdot \delta_t e^{\bar{v}}) = d(\gamma(0), \gamma(t)) \leq \int_0^t
\xi_{\varepsilon}(\gamma(s), \dot{\gamma}(s)) \,{\rm d}s.$$
Now, we are in position to apply Lemma \ref{firstlemma} to the metric $\xi_{\varepsilon}$:
for every fixed $a \in \Omega$, a.e.\ $x \in \Omega$ and all $v \in H_x \G$
\[ \big| \langle \nabla_{\mathbb{G}}d_a(x), v \rangle_x \big|
\leq  \liminf_{t \rightarrow 0}\frac{d(x, x \cdot \delta_t e^{\bar{v}})}{t} 
\leq  \limsup_{t \rightarrow 0}\frac{d(x, x \cdot \delta_t e^{\bar{v}})}{t} \leq \xi_{\varepsilon}(x, v).\]
Taking the least upper bound w.r.t.\ $a \in \Theta$ and letting $\varepsilon \rightarrow 0$,
we obtain 
$$\xi(x, v) \leq  \liminf_{t \rightarrow 0} \frac{d(x, x \cdot\delta_t e^{\bar{v}})}{\lvert t \rvert}
\leq \limsup_{t \rightarrow 0} \frac{d(x, x \cdot \delta_t e^{\bar{v}})}{\lvert t \rvert}
\leq \xi(x, v).$$
This proves the convexity of the limit, i.e.\ of the metric derivative on the horizontal bundle.
\end{proof}

\section{Application I: \texorpdfstring{$\Gamma$}{Gamma}-convergence}
We start this section by briefly recalling the notion of $\Gamma$-convergence and we refer the interested reader to \cite{DalMaso} for a complete overview on the subject.
Let $(M, \tau)$ be a topological space satisfying the first axiom of countability. A sequence of maps $F_h:M\to \overline{\mathbb{R}}$ is said to $\Gamma(\tau)$-converge to $F$ and we will write $F_h \xrightarrow{\Gamma(\tau)} F$ ( or simply $F_h \xrightarrow{\Gamma} F$ if there is no risk of confusion) if the following two conditions hold:
\begin{itemize}
\item[] $(${\rm $\Gamma$-$\liminf$ inequality}$)$ for any $x\in M$ and for any sequence $(x_h)_h$ converging to $x$ in $M$ one has
\[
F(x)\le \liminf_{h\to\infty} F_h(x_h)\,;
\]
\item[] $(${\rm $\Gamma$-$\limsup$ inequality}$)$ for any $x\in M$, there exists a sequence $(x_h)_h$ converging to $x$ in $M$ such that 
\[
\limsup_{h\to\infty} F_h(x_h)\le F(x)\,.
\]
\end{itemize}
To any distance $d$ in $\mathcal{D}_{cc}(\Omega)$, we are going to associate some functionals 
defined respectively on the class $\mathcal{B}(\Omega)$ of all positive and 
finite Borel measures $\mu$ on $\Omega \times \Omega$ and on $ \Lip([0,1],\Omega)$.
We set 
\begin{align*}
J_d(\mu) &= \int d(x, y)\,{\rm d} \mu(x, y), \;\;\, \mu \in \mathcal{B}(\Omega) ; \\
L_d(\gamma) &= \int_0^1 \varphi_d(\gamma(t), \dot{\gamma}(t))\,{\rm d}t, \;\;\, 
\gamma \in \Lip([0,1],\Omega).
\end{align*}
As already mentioned in Subsection \ref{s:subFins_duality}, we equip
$\mathcal D_{cc}(\Omega)$ with the topology of the uniform convergence
on compact subsets of $\Omega\times\Omega$. Moreover, we endow
$\mathcal B(\Omega)$ and ${\rm Lip}([0,1],\Omega)$ with the topology
of weak$^*$ convergence and of the uniform convergence, respectively.
\medskip

The following result is strongly inspired by \cite[Theorem 3.1]{BDF}.
\begin{thm}\label{gammaconvergence} Let $\Omega \subset \mathbb{G}$ be an open set in a Carnot group of step $k$ and let $(d_n)_n$ and $d$ belong to $\mathcal{D}_{cc}(\Omega)$.
If $J_n, L_n$ and $J, L$ are the functionals associated respectively to $d_n$ and $d$, defined as before, then the following conditions are equivalent:
\begin{itemize}
\item[(i)] $d_n \rightarrow d$\, in $\mathcal{D}_{cc}(\Omega)$;
\item[(ii)] $J_n \xrightarrow{\Gamma} J$ on $\mathcal{B}(\Omega)$;
\item[(iii)] $L_n \xrightarrow{\Gamma} L$ on $\Lip([0,1],\Omega)$.
\end{itemize}
Moreover, if \(\Omega\) is bounded, then $\rm (i)$, $\rm (ii)$  and $\rm (iii)$ 
are equivalent to the following condition:
\begin{itemize}
\item[(iv)] $J_n$ continuously converges to $J$, meaning that
$J(\mu)=\lim_n J_n(\mu_n)$ holds whenever the sequence
$(\mu_n)_n\subset\mathcal B(\Omega)$ weakly$^*$ converges to
$\mu\in\mathcal B(\Omega)$.
\end{itemize}
\end{thm}
\proof 
(i) $\Rightarrow$ (ii). In order to prove the $\Gamma$-lim inf inequality,
fix $\mu\in\mathcal B(\Omega)$ and $(\mu_n)_n\subset\mathcal B(\Omega)$ such
that $\mu_n$ weakly$^*$ converges to $\mu$. Fix a sequence $(\eta_k)_k$ of
compactly-supported continuous functions $\eta_k:\Omega\times\Omega\to[0,1]$
such that $\eta_k(x)\nearrow 1$ for every $x\in\Omega$. Since $d_n\to d$ in
$\mathcal D_{cc}(\Omega)$, we deduce that for any $k\in\mathbb N$ we have that
$\eta_k d_n\to\eta_k d$ uniformly as $n\to\infty$, thus there exists a sequence
$(\varepsilon^k_n)_n\subset(0,+\infty)$ such that $\varepsilon^k_n\searrow 0$
as $n\to\infty$ and $ \lvert \eta_k d_n-\eta_k d \rvert \leq\varepsilon^k_n$ on $\Omega\times\Omega$.
Moreover, since $\mu_n$ weakly$^*$ converges to $\mu$, by using Banach--Steinhaus
Theorem we deduce that $\sup_n\mu_n(\Omega\times\Omega)<+\infty$.
Therefore, since $\eta_k d$ is continuous and bounded, we get that
\[\begin{split}
&\bigg|\int\eta_k(x,y)d_n(x,y)\,{\rm d}\mu_n(x,y)-\int\eta_k(x,y)d(x,y)\,{\rm d}\mu(x,y)\bigg|\\\leq\,&\varepsilon^k_n\,\mu_n(\Omega\times\Omega)+
\bigg|\int\eta_k(x,y)d(x,y)\,{\rm d}\mu_n(x,y)-\int\eta_k(x,y)d(x,y)\,{\rm d}\mu(x,y)\bigg|
\to 0\quad\mbox{as }n\to\infty,
\end{split}\]
for every $k\in\mathbb N$. In particular, for any $k\in\mathbb N$ we have that
\[
\int\eta_k(x,y)d(x,y)\,{\rm d}\mu(x,y)=\lim_{n\to\infty}\int\eta_k(x,y)d_n(x,y)
\,{\rm d}\mu_n(x,y)\leq\liminf_{n\to\infty}J_n(\mu_n).
\]
By monotone convergence theorem, we conclude that $J(\mu)\leq\liminf_n J_n(\mu_n)$,
as desired.

Let us pass to the verification of the $\Gamma$-lim sup inequality.
Fix any $\mu\in\mathcal B(\Omega)$. We aim to show that the sequence
constantly equal to $\mu$ is a recovery sequence, namely
$J(\mu)\geq\limsup_n J_n(\mu)$. If $J(\mu)=+\infty$, then there is nothing
to prove. Thus suppose that $J(\mu)<+\infty$.\\
Since $( 1/\alpha) d_{cc}\leq d$,
we deduce that $d_{cc}\in L^1(\mu)$. By combining this information with
the fact that $d_n\leq\alpha d_{cc}$ for all $n\in\mathbb N$ and $d_n\to d$
pointwise on $\Omega\times\Omega$, we are in a position to apply the dominated
convergence theorem, obtaining that $J(\mu)=\int d(x,y)\,{\rm d}\mu(x,y)=
\lim_n\int d_n(x,y)\,{\rm d}\mu(x,y)=\lim_n J_n(\mu)$.\\
(i) $\Rightarrow$ (iii). For every $\gamma \in \Lip([0,1],\Omega)$, we have to prove the following two claims:
\begin{align}
\label{eq4}
&\forall \; \gamma_n \xrightarrow{} \gamma \;\; : \;\;
L_d(\gamma) \leq \liminf_{n \rightarrow \infty} L_{d_n}(\gamma_n),\\
\label{eq5}
&\exists \; \gamma_n \xrightarrow{} \gamma \;\; : \;\;
L_d(\gamma) \geq \limsup_{n \rightarrow \infty} L_{d_n}(\gamma_n).
\end{align}
We begin proving \eqref{eq4}. Let $\gamma_n \rightarrow \gamma$ in $\Lip([0,1],\Omega)$. By definition of $L_d(\gamma)$, for any $\delta \geq 0$ we can find a partition of $[0,1]$, indexed over a finite set $I_{\delta}$, such that
\begin{equation}
\label{eq3}
L_d(\gamma) \leq \delta + \sum_{i \in I_{\delta}} d(\gamma(t_i), \gamma(t_{i + 1})).
\end{equation}
Since $\{\gamma_n\}_{n \in \mathbb{N}}$ converges uniformly on $[0,1]$, we may assume that 
$$\exists \;\,\bar{n} \in \mathbb{N} \;\; :\;\; (\gamma_n(s), \gamma_n(t)) \in K \subset \Omega \times \Omega, \;\; \forall s,t \in [0,1], \;\;\forall n \geq \bar{n},$$
where $K$ is compact. Then, for every $i \in I_{\delta}$,
\begin{align*}
&\lvert d_n( \gamma_n(t_i), \gamma_n(t_{i + 1})) - d(\gamma(t_i), \gamma(t_{i + 1})) \rvert\\ \leq  &\,\lvert d_n(\gamma_n(t_i), \gamma_n(t_{i + 1})) - d(\gamma_n(t_i), \gamma_n(t_{i + 1})) \rvert
+ \lvert d(\gamma_n(t_i), \gamma_n(t_{i + 1})) - d(\gamma(t_i), \gamma(t_{i + 1})) \rvert \\
\leq\,&\sup_K \lvert d_n - d \rvert + d(\gamma(t_i),\gamma_n(t_i))+d(\gamma(t_{i+1}),\gamma_n(t_{i+1}))\\
\leq\,&\sup_K \lvert d_n - d \rvert +\alpha \Bigl( d_{cc}(\gamma(t_i),\gamma_n(t_i))+ d_{cc}(\gamma(t_{i+1}),\gamma_n(t_{i+1}))
\Bigr)\\
\leq\,&\sup_K \lvert d_n - d \rvert +\alpha C_{K'} \Bigl( \lvert \gamma(t_i)-\gamma_n(t_i) \rvert^{\frac{1}{k}}+\lvert \gamma(t_{i+1})-\gamma_n(t_{i+1}) \rvert^{\frac{1}{k}} \Bigr)\\
\leq\,&\sup_K \lvert d_n - d \rvert +2\alpha C_{K'} \sup_{[0,1]}\lvert \gamma-\gamma_n \rvert^{\frac{1}{k}}=:\xi_n,
\end{align*}
where $K'\subset\Omega$ is any compact set such that $K\subset K'\times K'$
and $C_{K'}$ is the constant provided by Theorem \ref{sea}.
Note that $\xi_n\to 0$.
We infer from \eqref{eq3} and from the definition of $L_{d_n}$ that
$$L_d(\gamma) \leq \delta + \sum_{i \in I_{\delta}} [d_n(\gamma_n(t_i), \gamma_n(t_{i + 1})) + \xi_n] \leq \delta + L_{d_n}(\gamma_n) + \xi_n \card(I_{\delta}).$$
Passing to the lim inf as $n \rightarrow + \infty$, we get
$$L_d(\gamma) \leq \liminf_{n \rightarrow \infty} L_{d_n}(\gamma_n) + \delta.$$
This yields \eqref{eq4} by the arbitrariness of $\delta > 0$.\\

We prove now \eqref{eq5}. Let $\gamma \in \Lip(\Omega)$, let $K \subset \Omega \times \Omega$ compact be chosen as above, and let $r(n) \rightarrow \infty$ be a sequence such that
$$\lim_{n \rightarrow \infty} r(n)\, \sup_K \lvert d_n - d \rvert = 0.$$
For every $n \in \mathbb{N}$, let $I_n$ be the partition of $[0,1]$ into $r(n)$ intervals of equal length, and denote by $\{t_n^i\}, i=1, \ldots, r(n) + 1$, the endpoints of such intervals.
Let $\gamma_n$ be a curve whose restriction $\gamma_n^i$ to the interval $[t_n^i, t_n^{i + 1}]$ is defined by
\begin{equation}\label{construction}
\gamma_n^i(t_n^i) = \gamma(t_n^i), \;\; \gamma_n^i(t_n^{i + 1}) = \gamma(t_n^{i + 1}),
\;\;\; L_{d_n}(\gamma_n^i) \leq d_n(\gamma(t_n^i), \gamma(t_n^{i + 1})) + \frac{1}{2^{r(n)}}.
\end{equation}
\textbf{Claim:} The sequence $(\gamma_n)_{n \in \mathbb{N}}$ converges to $\gamma$ in $\Lip([0,1],\Omega)$. \\

Let us prove the claim. Fix a compact set \(K\subset\Omega\) such that
$\gamma(t),\gamma_n(t)\in K$ for every $n\in\mathbb N$ and $t\in[0,1]$.
Given any $n\in\mathbb{N}$ and $t\in(0,1]$, we denote by $(t_n^{-}, t_n^{+}]$
the interval of $I_n$ containing $t$. Consider the constant $C_K$ given by
Theorem \ref{sea}. Then it holds that
$$ \frac{1}{C_K} \lvert \gamma_n(t) - \gamma(t) \rvert \leq d_{cc}(\gamma_n(t), \gamma(t)) \leq d_{cc}(\gamma_n(t), \gamma_n(t_n^{+})) +
d_{cc}(\gamma(t_n^{+}), \gamma(t)) \eqqcolon A_n + B_n.$$
Now, by the uniform continuity of $\gamma$ on $[0,1]$, the term $B_n$ tends to zero as 
$n \rightarrow + \infty$ uniformly with respect to $t$. The same holds for $A_n$, since (using Lemma \ref{estimate}) it can be estimated as
\begin{align*}
\frac{1}{\alpha C_K} \lvert \gamma_n(t) - \gamma_n(t_n^+) \rvert &\leq 
\frac{1}{\alpha} d_{cc}(\gamma_n(t), \gamma_n(t_n^+)) \leq d_n(\gamma_n(t), \gamma_n(t_n^+))
\leq L_{d_n} (\gamma_n|_{[t, t_n^{+}]}) \\
&\leq L_{d_n} (\gamma_n|_{[t_n^{-}, t_n^{+}]}) \leq 
\alpha d_{cc}(\gamma_n(t_n^{-}), \gamma_n(t_n^+)) + \frac{1}{2^{r(n)}} \\
&= \alpha d_{cc}(\gamma(t_n^{-}), \gamma(t_n^+)) + \frac{1}{2^{r(n)}}
\leq \alpha \cdot C_K \lvert \gamma(t_n^{-}) - \gamma(t_n^+) \rvert^{\frac{1}{k}} + \frac{1}{2^{r(n)}},
\end{align*}
where we used the continuity estimate (\ref{construction}) and the fact that $d_n \in \mathcal{D}_{cc}(\Omega)$.
Now, by definition of $L_d(\gamma)$ and the construction (\ref{construction}), we infer that
\begin{align*}
L_d(\gamma) &\geq \sum_{i = 1}^{r(n)} d(\gamma(t_n^i), \gamma(t_n^{i + 1})) \\
&= \sum_{i = 1}^{r(n)} d_n(\gamma(t_n^i), \gamma(t_n^{i + 1}))
+ \sum_{i = 1}^{r(n)} [ d(\gamma(t_n^i), \gamma(t_n^{i + 1})) - d_n(\gamma(t_n^i), \gamma(t_n^{i + 1}))]  \\
&\geq L_{d_n}(\gamma_n) - \frac{r(n)}{2^{r(n)}} + \sum_{i = 1}^{r(n)} [ d(\gamma(t_n^i), \gamma(t_n^{i + 1})) - d_n(\gamma(t_n^i), \gamma(t_n^{i + 1}))].
\end{align*}
To get the required inequality, it is enough to pass to the limsup in the above inequality, noticing that, by the choice of the sequence $r(n)$,
we have
$$\lim_{n \rightarrow + \infty} \sum_{i = 1}^{r(n)} [ d(\gamma(t_n^i), \gamma(t_n^{i + 1})) - d_n(\gamma(t_n^i), \gamma(t_n^{i + 1}))] \leq
\lim_{n \rightarrow + \infty} r(n) \sup_K \lvert d_{n} - d \rvert = 0,$$
and then we get the desired conclusion (\ref{eq5}).\\
(iii) $\Rightarrow$ (i). This implication follows from the following fact:\\
 
\textbf{Claim:} The class $\mathcal{D}_{cc}(\Omega)$ is compact.\\

As we are going to show, the above claim is obtained as a consequence
of the Ascoli--Arzel\'a Theorem and the implication (i) $\Rightarrow$ (iii) already proved.
Let $(d_n)_n\subset\mathcal D_{cc}(\Omega)$ be a given sequence.
First of all, for any $(x,y)\in\Omega\times\Omega$ we have that $(d_n(x,y))_n$
is a bounded sequence, as granted by the following estimate:
\begin{equation}
d_n(x, y) \leq \alpha d_{cc}(x, y) \quad\text{for every } x,y \in \Omega \quad\text{and } n \in \mathbb{N}.
\end{equation}
Moreover, we have to prove that the sequence $(d_n)_n \in \mathcal{D}_{cc}(\Omega)$ is equi-continuous, in other words, that
for every $x, x', y, y' \in \Omega \subset \mathbb{G}$ it holds
$$\forall \varepsilon > 0 \;\; \exists\,\delta > 0 \; : 
\begin{cases}
\lvert x' - x \rvert < \delta \\
\lvert y' - y \rvert < \delta        
\end{cases}
\Rightarrow \; \lvert d_n(x, x') - d_n(y, y')\rvert < \varepsilon, \;\; \forall n \in \mathbb{N}.$$
By using the triangle inequality and Theorem \ref{sea}, we obtain that
\begin{align*}
\lvert d_n(x, y) - d_n(x', y')\rvert &\leq d_n(x, x') + d_n(y, y')
 \leq \alpha \bigl( d_{cc}(x, x') + d_{cc}(y, y') \bigr)\\
& \leq \alpha C_K\bigl( \lvert x' - x\rvert^{\frac{1}{k}} + \lvert y'-y \rvert^{\frac{1}{k}} \bigr) .
\end{align*}
Choosing $\delta = 2\frac{\epsilon^k}{C_K\beta}$, we obtain
\begin{equation*}
\lvert d_n(x, y) - d_n(x', y')\rvert \leq  \varepsilon.
\end{equation*}
Hence, we may extract a subsequence converging to some element $d$ in $\mathcal{D}_{cc}(\Omega)$.

To prove that $d$ is geodesic, we use the implication (i) $\Rightarrow$ (iii), which ensures that 
$L_n \xrightarrow{\Gamma} L$.
Fix $x,y\in\Omega$.
We will prove that we have the $\Gamma$-convergence for the modified functionals:
$$ 
\tilde{L}_n(\gamma)\coloneqq
\begin{cases}
L_n(\gamma),\quad\mbox{if }\gamma(0) = x \mbox{ and }\gamma(1)=y; \\
+\infty,\quad\mbox{otherwise};\\
\end{cases} 
$$
\[ \tilde{L}(\gamma)\coloneqq
\begin{cases}
L(\gamma),\quad\mbox{if }\gamma(0) = x\mbox{ and }\gamma(1)=y;\\
+\infty,\quad\mbox{otherwise}.\\
\end{cases} 
\]
Arguing as in \cite[Theorem 3.1]{BDF}, we can show that $\liminf_{n \rightarrow \infty} \tilde{L}_n(\gamma_n) \geq \tilde{L}(\gamma)$  whenever $\gamma_n \rightarrow \gamma$ in
$\Lip(\Omega)$.
To conclude we need to prove that, for every $\gamma \in \Lip([0,1],\Omega)$,
there exists an approximating sequence $\{\tilde{\gamma}_n\}$ satisfying 
$\limsup_n\tilde{L}_n(\tilde{\gamma}_n) \leq \tilde{L}(\gamma)$.
We can assume without loss of generality that $\tilde{L}(\gamma) = L(\gamma)$.
Take a sequence $(\gamma_n)_{n \in \mathbb{N}}$ with $\gamma_n \rightarrow \gamma$ in $\Lip([0,1],\Omega)$ and since $L_n \xrightarrow{\Gamma} L$ we can suppose that
$\lim_n L_n(\gamma_n) = L(\gamma)$. One can construct the optimal sequence
modifying the curves as follows: 
\begin{align*}
\tilde{\gamma}_n(t) \coloneqq
\begin{cases}
\mbox{an almost }d_n\mbox{-geodesic connecting }x\mbox{ and }\gamma_n(\frac{1}{n}),
\quad\mbox{if }t \in [0, \frac{1}{n}];\\
\gamma_n(t),\quad\mbox{if }t \in [\frac{1}{n}, 1 {-}\frac{1}{n}];\\
\mbox{an almost }d_n\mbox{-geodesic connecting }\gamma_n(1 {- \frac{1}{n}})
\mbox{ and }y,\quad\mbox{if }t \in [1 {-} \frac{1}{n}, \frac{1}{n}].
\end{cases}
\end{align*}
Similarly to \cite[Theorem 3.1]{BDF}, and using again Lemma \ref{estimate} we get that $\tilde{\gamma}_n$ still converges to $\gamma$ in $\Lip(\Omega)$ and  \begin{align}\label{disss}
\frac{1}{\alpha} d_{cc} \left(\tilde{\gamma}_n(t), x \right)  \leq \alpha \cdot C_K \left[ \left|\gamma_n \left({\textstyle\frac{1}{n}}\right) - \gamma \left({\textstyle\frac{1}{n}}\right) \right| + \left|\gamma \left({\textstyle\frac{1}{n}}\right) - x \right| \right]^{\frac{1}{k}} 
+ \varepsilon_n
\end{align}
holds, where the last term tends to zero as $n \rightarrow \infty$, since $\gamma_n \rightarrow \gamma$ in $\Lip(\Omega)$.
It remains to show that $\limsup_n \tilde{L}_n(\tilde{\gamma}_n) \leq \tilde{L}(\gamma)$, indeed 
we have that
\begin{equation}
\label{eq7}
\tilde{L}_n(\tilde{\gamma}_n) \leq d_n(x, \gamma_n \left({\textstyle\frac{1}{n}}\right))
+ L_n(\gamma_n) + d_n(\gamma_n \left({\textstyle 1 - \frac{1}{n}}\right), y) + 2 \varepsilon_n.
\end{equation}
Notice now that, from (\ref{disss}), it follows in particular that
$\lim_n d_n(x, \gamma_n \left( \frac{1}{n}\right)) = 0$ and similarly
 $\lim_n d_n( \gamma_n \left(1 - \frac{1}{n}\right), y) = 0$. Hence passing to the lim sup as $n \rightarrow \infty$
in \eqref{eq7} gives
\[\limsup_{n \rightarrow \infty} \tilde{L}_n(\tilde{\gamma}) \leq \limsup_{n \rightarrow \infty} L(\gamma_n) = L(\gamma) = \tilde{L}(\gamma).\]
Thus, by the $\Gamma$-convergence of $L_n$ to $L$, we deduce that
\begin{equation}
\label{eq8}
\inf_{\gamma} \tilde{L}(\gamma) = \lim_{n \rightarrow \infty} \inf_{\gamma}
\tilde{L}_n(\gamma)
= \lim_{n \rightarrow \infty} d_n(x, y) = d(x, y).
\end{equation}
Since $d_n$ are geodesic distances in $\mathcal{D}_{cc}(\Omega)$, the equation \eqref{eq8} means exactly that $d$ is a geodesic distance, as desired.
\medskip

Finally, assume in addition that \(\Omega\) is bounded. On the one hand,
(iv) trivially implies (ii). On the other hand, we can prove that (i) implies (iv).
To this aim, fix any $\mu\in\mathcal B(\Omega)$ and $(\mu_n)_n\subset\mathcal B(\Omega)$
such that $\mu_n$ weakly$^*$ converges to $\mu$. Let \(\varepsilon>0\) be fixed.
We have that \(\sup_n\mu_n(\Omega\times\Omega)<+\infty\) by Banach--Steinhaus
Theorem. Moreover, we have that $\{\mu_n\}_n$ is weakly$^*$ relatively compact
by assumption, thus Prokhorov's Theorem yields the existence of a compact set
$K\subset\Omega\times\Omega$ such that $\mu_n((\Omega\times\Omega)\setminus K)
\leq\varepsilon$ for every $n\in\mathbb N$. Call $D$ the diameter of $\Omega$
with respect to $d_{cc}$. Since $d:\Omega\to\Omega\to\mathbb R$
is bounded and continuous, we deduce that
\[\begin{split}
\big|J_n(\mu_n)-J(\mu)\big|&\leq\int_K \lvert d_n-d \rvert\,{\rm d}\mu_n+
\int_{(\Omega\times\Omega)\setminus K} \lvert d_n-d \rvert\,{\rm d}\mu_n+
\big|J(\mu_n)-J(\mu)\big|\\
&\leq\mu_n(\Omega\times\Omega)\max_K|d_n-d|+2\beta D\varepsilon+
\bigg|\int d\,{\rm d}\mu_n-\int d\,{\rm d}\mu\bigg|,
\end{split}\]
whence by letting \(n\to\infty\) we get
$\limsup_n \lvert J_n(\mu_n)-J(\mu) \rvert \leq 2\beta D\varepsilon\).
By arbitrariness of $\varepsilon$, we finally conclude that
$J(\mu)=\lim_n J_n(\mu_n)$, so that $(iv)$ is proved.
\endproof

\section{Applications II: Intrinsic geometry and sub-Finsler structure}

The present section is devoted to generalizing the metric results contained in \cite{GUO}. To this aim, we introduce two distances which involve the structure of the sub-Finsler metric.

\begin{defn} If $\varphi \in \mathcal{M}_{cc}^\alpha(\mathbb{G})$ is a sub-Finsler convex metric, for
every \(x,y\in\G\) we define the following quantity:
\begin{equation}\label{lipdist}
\delta_{\varphi}(x, y) \coloneqq \sup\big\{ \lvert f(x) - f(y) \rvert \, \big| \,f\colon\G\rightarrow\mathbb R\emph{ Lipschitz},\, 
\norm{\varphi(\cdot, \nabla_{\mathbb{G}} f (\cdot))}_{\infty} \leq 1 \big\}.
\end{equation}
\end{defn}
Recall that Pansu's Theorem assures that $\nabla_{\mathbb{G}} f(x)$ exists at almost every $x \in \mathbb{G}$ and thus
the above definition makes sense. From now on, we will say that any Lipschitz function satisfying the conditions in 
\eqref{lipdist} is a \emph{competitor} for $\delta_{\varphi}$. Moreover, we have that
\begin{lem} $\delta_{\varphi} : \mathbb{G} \times \mathbb{G} \rightarrow [0, + \infty)$ is a distance.
\end{lem}
\begin{proof} Clearly, we have that $\delta_{\varphi}(x, y) \geq 0$
for every $
x,y\in\mathbb{G}$ and $\delta_\varphi(x,y)>0$ if $x \neq y$.
The symmetry comes from the fact that 
$ \lvert f(x) - f(y) \rvert = \lvert f(y) - f(x) \rvert$. Also,
$\delta_{\varphi}$ satisfies the triangle inequality since for every $x, y, z \in \mathbb{G}$ we have
$\delta_{\varphi}(x, y) + \delta_{\varphi}(y,z) \geq \lvert f(x) - f(y) \rvert + \lvert f(y) - f(z)\rvert 
\geq \lvert f(x) - f(z) \rvert$. Passing to the supremum
on the right-hand side for every $f$ Lipschitz function such that $\norm{\varphi(x, \nabla_{\mathbb{G}} f (x))}_{\infty} \leq 1$, we get that $\delta_{\varphi}(x, y) + \delta_{\varphi}(y,z) \geq \delta_{\varphi}(x, z)$.
\end{proof}
Under some assumptions, we will show in Theorem \ref{maintheorem} that $\delta_\varphi$ turns out to be a distance in $\mathcal{D}_{cc}(\G)$. 
\begin{defn}
The \emph{pointwise Lipschitz constant} of a Lipschitz function $f : \mathbb{G} \rightarrow \mathbb{R}$ is defined as
$$ \Lip_{\delta_{\varphi}} f(x) = \limsup_{y \rightarrow x} \frac{ \lvert f(y) - f(x) \rvert}{\delta_{\varphi}(x, y)}
\quad\mbox{for every }x\in\G.$$
\end{defn}
We recall now the notion of intrinsic distance that was introduced by De Cecco--Palmieri in \cite[Definition 1.4]{DCP4} in the context of Lipschitz manifold.

\begin{defn}\label{intrinsic}
Given any $\varphi \in \mathcal{M}_{cc}^\alpha(\mathbb{G})$, we define its induced \emph{intrinsic distance} $d_\varphi$ as
\[
d_\varphi(x,y)\coloneqq\inf_\gamma\int_0^1\varphi(\gamma(t),\dot\gamma(t))  
\,{\rm d}t\quad\text{ for every }x,y\in\G,
\]
where the infimum is taken over all horizontal curves $\gamma \in \mathcal{H}([0,1], \mathbb{G})$ joining $x$ and $y$.
\end{defn}
The quantity \(d_\varphi(x,y)\) is well-defined because the map $t \mapsto (\gamma(t), \dot{\gamma}(t))$ is Borel
measurable on the horizontal bundle. \\
Let us observe that in Definition \ref{defF} we are not requiring any regularity on $\varphi$ besides its Borel measurability.
At this level of generality (namely, without semicontinuity assumptions) $d_\varphi$ might
exhibit some `pathological' behaviour, as we can see in the following example. 
Others examples of  geodesic distances, contained in $\mathcal{D}_{cc}(\mathbb{R}^2)$, which are not intrinsics can be found in \cite[Example 1.8]{rif14} and 
\cite[Corollary 3.4]{rif15}.
\begin{ex}
{ \rm Let $\mathbb R^2$ with the Euclidean structure and consider the Borel set $N\subset\mathbb R^2$ as 
$$ N\coloneqq\bigcup_{x,y\in\mathbb Q^2}S_{x,y},$$ 
where $S_{x,y}$ stands for
the segment joining $x$ and $y$. Notice that $N$ is $\mathcal L^2$-negligible and we define the metric
$\varphi\colon\mathbb R^2\times\mathbb R^2\to[0,+\infty)$ as
\[
\varphi(x,v)\coloneqq\left\{\begin{array}{ll}
\lvert v \rvert,\\
2 \lvert v \rvert,
\end{array}\quad\begin{array}{ll}
\text{ if }x\in N,\\
\text{ if }x\notin N.
\end{array}\right.
\]
Since $\varphi\in\mathcal M_{cc}^2(\mathbb R^2)$, for every $x,y\in\mathbb R^2$ it holds that $|x-y|\leq d_\varphi(x,y)\leq 2|x-y|$.
In particular $d_\varphi\colon\mathbb R^2\times\mathbb R^2\to[0,+\infty)$ is continuous when the domain is endowed with the Euclidean distance.
Now observe that for any $x,y\in\mathbb Q^2$ we have that $d_\varphi(x,y)=|x-y|$, the shortest path being exactly the segment $S_{x,y}$. By continuity
of $d_\varphi$ and thanks to the density of $\mathbb Q^2$ in $\mathbb R^2$, we conclude that $d_\varphi(x,y)=|x-y|$ for every $x,y\in\mathbb R^2$.
This shows that, even if $\varphi(x,\cdot)$ is equal to $2|\cdot|$ for $\mathcal L^2$-a.e.\ $x\in\mathbb R^2$, the distance $d_\varphi$ coincides
with the Euclidean distance. In other words, the behaviour of $\varphi$ on the null set $N$ completely determines the induced distance $d_\varphi$.}
\end{ex}

A further important concept for our treatment is the classical notion of Finsler metrics on Carnot 
groups.
\begin{defn}\label{intrinsic_F}
We say that a map $F : T \mathbb{G} \rightarrow [0, + \infty)$ is a 
\emph{Finsler metric}
if
the following properties hold:
\begin{itemize}
\item{} $F$ is continuous on $T\G$ and smooth on $ T \mathbb{G} \setminus \{0\}$,
\item{} 
the Hessian matrix of $F^2$ is positive definite for any vector $v \in T_x \mathbb{G} \setminus \{0\}$
for every $x\in\G$.
\end{itemize} 
Moreover, we denote by $d_F$ the length distance on $\G$ induced
by $F$, namely we set
\[
d_{F}(x, y):=\inf_\gamma\int_0^1 F(\gamma(t),\dot{\gamma}(t)) \,{\rm d}t \quad\mbox{for every }x,y\in\G,
\]
where the infimum is taken among all curves $\gamma\in{\rm Lip}([0,1],\G)$
joining $x$ and $y$.
\end{defn}

Let us observe that the intrinsic distance is induced by a metric on the horizontal bundle while the latter comes from a metric defined on the entire tangent bundle.
 
\begin{lem} If $\varphi \in \mathcal{M}_{cc}^\alpha(\mathbb{G})$ is a sub-Finsler convex metric, 
then $d_{\varphi}$ 
is a geodesic distance belonging to $\mathcal{D}_{cc}(\mathbb{G})$.
\end{lem}
\proof Since $(\G, d_\varphi)$ is a complete, locally compact length space, then $d_{\varphi}$
is a geodesic distance, thanks to the general result contained in \cite[Theorem 2.5.23]{Burago}. \\
To prove the claim, we have that $d_{\varphi}(x, y) \geq 0$ for every $x, y \in \mathbb{G}$ since the integral of $\varphi(\gamma(\cdot),\dot\gamma(\cdot))$ is non-negative.
In order to prove the symmetry, let us consider $\gamma \in \mathcal{H}([0,1],\mathbb{G})$ such that $\gamma(0) = x$ and $\gamma(1) = y$.
Set
$\xi : [0, 1] \rightarrow \mathbb{G}$ as $\xi(t) = \gamma(1 - t)$, hence this is a horizontal curve in $[0,1]$.
By the 1-homogeneity of $\varphi(x,\cdot)$, we get that
\begin{align*}
\int_0^1 \varphi(\xi(t), \dot{\xi}(t)) \,{\rm d}t
 &= \int_0^1 \varphi(\gamma(1 - t) , - \dot{\gamma}(1 - t)) \,{\rm d}t 
= \int_0^1 \varphi(\gamma(s), - \dot{\gamma}(s)) \,{\rm d}s \\
&=  \int_0^1 \varphi(\gamma(s), \dot{\gamma}(s)) \,{\rm d}s.
\end{align*}
So now, passing to the infimum over $\gamma \in  \mathcal{H}([0,1], \mathbb{G})$ we get that 
$d_{\varphi}(x, y) = d_{\varphi}(y, x)$.

To prove the triangle inequality, 
let $x, y, z \in \mathbb{G}$ and $\gamma_1,\gamma_2\in\mathcal{H}([0,1], \mathbb{G})$ be such that 
$\gamma_1(0) = x, \gamma_1(1)= y = \gamma_2(0)$, and $\gamma_2(1) = z$. 
Let us define the following curve: 
$$ 
\eta : [0,1] \rightarrow \mathbb{G},\quad \eta(t) \coloneqq
\begin{cases}
 \gamma_1(2 t)\quad\mbox{if }t \in [0, \frac{1}{2}]; \\
 \gamma_2(2 t - 1)\quad\mbox{if }t \in [\frac{1}{2}, 1].      \\
\end{cases} 
$$
Then we obtain that
\begin{align*}
d_{\varphi}(x, z) &\leq \int_0^1 \varphi(\eta(t), \dot{\eta}(t)) \,{\rm d}t = \int_0^{\frac{1}{2}} \varphi(\gamma_1(2 t), 2\dot{\gamma}_1(2 t)) \,{\rm d}t 
+ \int_{\frac{1}{2}}^1 \varphi(\gamma_2(2t - 1), 2\dot{\gamma}_2(2t - 1))\,{\rm d}t \\
&= \int_0^1 \varphi(\gamma_1(s), \dot{\gamma}_1(s)) \,{\rm d}s 
+ \int_0^1 \varphi(\gamma_2(s), \dot{\gamma}_2(s))\,{\rm d}s,
\end{align*}
where we applied a change-of-variable (in both integrals) and the $1$-homogeneity of $\varphi$.
Passing to the infimum respectively over all $\gamma_1,\gamma_2$, we conclude.
We are left to prove \eqref{ineq}. Let us take $x, y \in \mathbb{G}$ and consider the horizontal curve 
$ \gamma : [0,1] \rightarrow \mathbb{G}$ s.t.\ $\gamma(0) = x$ and $ \gamma(1) = y$.
By Proposition \ref{properties} we get that
\begin{align*}
\int_0^1 \varphi(\gamma(t), \dot{\gamma}(t))\,{\rm d}t
& \leq \alpha \int_0^1  \| \dot{\gamma}(t) \|_{\gamma(t)} \,{\rm d}t.
\end{align*} 
Thus, passing to the infimum in the right-hand side we obtain the conclusion and
the converse inequality can be achieved by arguing in a similar way.  
\endproof
\subsection{Main Results}
Before proving one of the main theorems, we recall some basic terminology. 
Given two Banach spaces \(\mathbb B_1\) and \(\mathbb B_2\),
and denoting with \({\rm L}(\mathbb B_1,\mathbb B_2)\) the space
of all linear and continuous operators \(T\colon\mathbb B_1\to\mathbb B_2\),
it holds that \({\rm L}(\mathbb B_1,\mathbb B_2)\) is a Banach space
if endowed with the usual pointwise operations and the operator norm,
namely
\[
\|T\|_{{\rm L}(\mathbb B_1,\mathbb B_2)}\coloneqq
\underset{v\in\mathbb B_1\setminus\{0\}}\sup\frac{\|T(v)\|_{\mathbb B_2}}
{\|v\|_{\mathbb B_1}}\quad\text{ for every }T\in{\rm L}(\mathbb B_1,\mathbb B_2).\]
\begin{rem}{\rm
Given a smooth map \(\varphi\colon M\to N\) between two smooth manifolds
\(M\), \(N\) and a point \(x\in M\), we denote by
\({\rm d}_x\varphi\colon T_x M\to T_{\varphi(x)}N\) the differential of
\(\varphi\) at \(x\).
We recall that if \(\gamma\colon[0,1]\to M\) is an
absolutely continuous curve in \(M\), then \(\sigma\coloneqq\varphi\circ\gamma\)
is an absolutely continuous curve in \(N\) and it holds that
\begin{equation}\label{eq:chain_rule_velocity}
\dot\sigma(t)={\rm d}_{\gamma(t)}\varphi[\dot\gamma(t)]
\quad\text{ for a.e.\ }t\in[0,1].
\end{equation}
We also point out that
\begin{equation}\label{eq:formula_exp_tv}
\frac{\rm d}{{\rm d}t}\delta_t e^v={\rm d}_e\tau_{\delta_t e^v}[v]
\quad\text{ for every }v\in H_e\G\text{ and }t\in(0,1).
\end{equation}
Indeed, calling \(\gamma\) the unique curve satisfying \eqref{eq:ODE_def_exp}
and defining \(\gamma^t(s)\coloneqq\gamma(ts)\) for all \(t\in(0,1)\) and
\(s\in[0,1]\), we may compute
\[
\frac{\rm d}{{\rm d}s}\gamma^t(s)=t\dot\gamma(ts)=
t\,{\rm d}_e\tau_{\gamma(ts)}[v]={\rm d}_e\tau_{\gamma^t(s)}[tv]
\quad\text{ for every }s\in(0,1),
\]
which shows that \(\gamma^t\) fulfills the ODE defining \(tv\),
so that \eqref{eq:exp_on_horizontal} yields
\(\gamma(t)=\gamma^t(1)=e^{tv}=\delta_t e^v\) for every \(t\in(0,1)\)
and accordingly the identity claimed in \eqref{eq:formula_exp_tv} is proved.
\fr}
\end{rem}

In general, let us observe that, if $\psi$ is a sub-Finsler metric, then the metric derivative with  $d = \delta_\psi$, namely $\varphi_{\delta_{\psi}}$, could be very different from $\psi$ 
(see \cite[Example 1.5]{rif14}).
Our purpose is to show a different result for the intrinsic distances.
It tells us that, given a sub-Finsler convex metric $\psi$, the metric derivative with respect to $d_\psi$ is bounded above by $\psi$ almost everywhere.
Moreover, we show that the equality holds, for instance, when $\psi$ is lower semicontinuous.

\begin{thm}\label{inequality}\label{uguaglianza} 
Let $\psi \in \mathcal{M}_{cc}^\alpha(\mathbb{G})$ be a sub-Finsler convex metric.
Then the following properties are verified:
\begin{itemize}
    \item[\(\rm i)\)] It holds that
    \[
    \text{for a.e.\ }x\in\G,\quad\varphi_{d_\psi}(x,v)\leq\psi(x,v)
    \quad\text{ for every }v\in H_x\G.
    \]
    \item[\(\rm ii)\)] If \(\psi\) is upper semicontinuous, then
    \[
    \varphi_{d_\psi}(x,v)\leq\psi(x,v)\quad\text{ for every }(x,v)\in H\G.
    \]
    \item[\(\rm iii)\)] If \(\psi\) is lower semicontinuous, then
    \[
    \varphi_{d_\psi}(x,v)\geq\psi(x,v)\quad\text{ for every }(x,v)\in H\G.
    \]
  In particular, for a.e.\ \(x\in\G\) it holds that
    \(\varphi_{d_\psi}(x,v)=\psi(x,v)\) for every \(v\in H_x\G\).
\end{itemize}
\end{thm}
\proof
\ \\
\boxed{i)} Given \(x\in\G\), \(v\in H_e\G\), and \(t>0\), we define
the curve \(\gamma=\gamma_{x,v,t}\colon[0,1]\to\G\) as
\[
\gamma(s)\coloneqq x\cdot\delta_{ts}e^v\quad\text{ for every }s\in[0,1].
\]
Notice that \(\gamma\) is horizontal and joins \(x\) to \(x\cdot\delta_t e^v\).
We can compute
\[
\dot\gamma(s)=\frac{\rm d}{{\rm d}s}\tau_x\big(\delta_{ts}e^{\bar v}\big)
\overset{\eqref{eq:chain_rule_velocity}}=
{\rm d}_{\delta_{ts}e^v}\tau_x\Big[\frac{\rm d}{{\rm d}s}
\delta_{ts}e^v\Big]\overset{\eqref{eq:formula_exp_tv}}=
{\rm d}_e\tau_{x\cdot\delta_{ts}e^v}[tv]\quad\text{ for every }s\in(0,1).
\]
Therefore, we may estimate
\begin{equation}\label{eq:est_d_psi_aux}\begin{split}
d_\psi(x,x\cdot\delta_t e^v)&\leq
\int_0^1\psi(\gamma(s),\dot\gamma(s))\,{\rm d}s
=t\int_0^1\psi\big(x\cdot\delta_{ts}e^v,
{\rm d}_e\tau_{x\cdot\delta_{ts}e^v}[v]\big)\,{\rm d}s\\
&=\int_0^t\psi\big(x\cdot\delta_s e^v,
{\rm d}_e\tau_{x\cdot\delta_se^v}[v]\big)\,{\rm d}s.
\end{split}\end{equation}
The next argument closely follows along the lines of Lemma
\ref{FubiniLeb}. Fix a dense sequence \((v_i)_i\) in the unit sphere of
\(H_e\G\) (w.r.t.\ the norm \(\|\cdot\|_e\)). Define \(v_i(x)\coloneqq
{\rm d}_e\tau_x[v_i]\) for every \(i\in\mathbb N\) and \(x\in\G\), so
that \((v_i(x))_i\) is a dense sequence in the unit sphere of \(H_x\G\)
(w.r.t.\ the norm \(\|\cdot\|_x\)). By using Lebesgue's differentiation
theorem and Fubini's theorem, we see that the set \(\Gamma_i\) of all couples
\((y,r)\in\G\times\mathbb R\) such that
\begin{equation}\label{eq:def_Gamma_i}
\psi\big(y\cdot\delta_r e^{v_i},{\rm d}_e
\tau_{y\cdot\delta_r e^{v_i}}[v_i]\big)=\lim_{t\searrow 0}
\frac{1}{t}\int_0^t\psi\big(y\cdot\delta_{r+s}e^{v_i},
{\rm d}_e\tau_{y\cdot\delta_{r+s}e^{v_i}}[v_i]\big)\,{\rm d}s
\end{equation}
has zero \(\mathcal L^{n+1}\)-measure. By using Fubini's theorem again,
we can find \(r\in\mathbb R\) such that for any \(i\in\mathbb N\) there
exists a \(\mathcal L^n\)-null set \(N_i\subset\G\) such that
\eqref{eq:def_Gamma_i} holds for every point \(y\in\G\setminus N_i\). Let us consider
the set \(N\coloneqq\bigcup_{i\in\mathbb N}\sigma_{\delta_r e^{v_i}}(N_i)\),
where \(\sigma_z\colon\G\to\G\) stands for the right-translation map
\(\sigma_z w\coloneqq w\cdot z\). The right-invariance of \(\mathcal L^n\)
grants that \(N\) is \(\mathcal L^n\)-negligible. Given that
\begin{equation}\label{eq:Leb+Fub_aux}
\psi(x,v_i(x))=\lim_{t\searrow 0}\frac{1}{t}\int_0^t\psi
\big(x\cdot\delta_s e^{v_i},{\rm d}_e\tau_{x\cdot\delta_s e^{v_i}}[v_i]\big)
\,{\rm d}s\quad\text{ for every }i\in\mathbb N\text{ and }x\in\G\setminus N,
\end{equation}
we can  conclude that
\[\begin{split}
\varphi_{d_\psi}(x,v_i(x))&\overset{\phantom{\eqref{eq:Leb+Fub_aux}}}=
\lim_{t\searrow 0}\frac{d_\psi(x,x\cdot\delta_t e^{v_i})}{t}
\overset{\eqref{eq:est_d_psi_aux}}\leq
\lim_{t\searrow 0}\frac{1}{t}\int_0^t\psi\big(x\cdot\delta_s e^{v_i},
{\rm d}_e\tau_{x\cdot\delta_s e^{v_i}}[v_i]\big)\,{\rm d}s\\
&\overset{\eqref{eq:Leb+Fub_aux}}=\psi(x,v_i(x))\quad\text{ for every }
i\in\mathbb N\text{ and }x\in\G\setminus N.
\end{split}\]
Since \(\psi(x,\cdot)\) is continuous and positively \(1\)-homogeneous,
and \((v_i(x))_i\) is dense in the unit \(\|\cdot\|_x\)-sphere of \(H_x\G\),
we deduce that \(\varphi_{d_\psi}(x,w)\leq\psi(x,w)\) for every
\(x\in\G\setminus N\) and \(w\in H_x\G\).\\
\boxed{ii)} Suppose \(\psi\) is upper semicontinuous. Let
\((x,v)\in H\G\) be fixed. Given any \(\varepsilon>0\), we can thus find
\(t_\varepsilon>0\) such that, setting
\(\bar v\coloneqq{\rm d}_x\tau_{x^{-1}}[v]\) for brevity, it holds that
\begin{equation}\label{eq:usc_psi_aux}
\psi\big(x\cdot\delta_t e^{\bar v},{\rm d}_e\tau_{x\cdot\delta_t e^{\bar v}}
[\bar v]\big)\leq\psi(x,v)+\varepsilon\quad
\text{ for every }t\in(0,t_\varepsilon).
\end{equation}
In particular, we may estimate
\[
\varphi_{d_\psi}(x,v)=
\lim_{t\searrow 0}\frac{d_\psi(x,x\cdot\delta_t e^{\bar v})}{t}
\overset{\eqref{eq:est_d_psi_aux}}\leq\lim_{t\searrow 0}\frac{1}{t}\int_0^t
\psi\big(x\cdot\delta_s e^{\bar v},{\rm d}_e\tau_{x\cdot\delta_s e^{\bar v}}
[\bar v]\big)\,{\rm d}s
\overset{\eqref{eq:usc_psi_aux}}\leq\psi(x,v)+\varepsilon.
\]
Thanks to the arbitrariness of \(\varepsilon\), we can conclude that
\(\varphi_{d_\psi}(x,v)\leq\psi(x,v)\), as desired.\\
\boxed{iii)} Suppose \(\psi\) is lower semicontinuous. First of all, let us extend \(\|\cdot\|_e\) to a Hilbert norm (still denoted by
	\(\|\cdot\|_e\)) on the whole \(T_e\G=\g\), then by left-invariance
	we obtain a Hilbert norm \(\|\cdot\|_x\) on each tangent space \(T_x\G\).
	Throughout the rest of the proof, we
	assume that \(T_x\G\) is
	considered with respect to such norm \(\|\cdot\|_x\). Moreover, choose
	any norm \(\sf n : \mathfrak{g} \rightarrow [0, + \infty) \) on
	the Lie algebra which extends \(\psi(e,\cdot)\), so that
	\({\sf n}\leq\lambda\|\cdot\|_e\) for some \(\lambda>0\).\\
	Without loss of generality,
	up to replacing \(\psi\) with the translated metric \(\psi_x\), defined
	as \(\psi_x(y,v)\coloneqq\psi\big(x\cdot y,{\rm d}_y\tau_x[v]\big)\)
	for every \((y,v)\in H\G\), it is sufficient to prove the statement only
	for \(x=e\). Then let \(v\in H_e\G\) be fixed. For any \(t>0\) we have
	that the horizontal curve \([0,1]\ni s\mapsto\delta_{st}e^v\in\G\) is
	a competitor for \(d_\psi(e,\delta_t e^v)\), thus we may estimate
	\[\begin{split}
	d_\psi(e,\delta_t e^v)&\leq
	\int_0^1\psi(\delta_{st}e^v,t\,{\rm d}_e\tau_{\delta_{st}e^v}[v])\,{\rm d} s
	=t\int_0^1\psi(\delta_{st}e^v,{\rm d}_e\tau_{\delta_{st}e^v}[v])\,{\rm d}s\\
	&\leq\alpha t\int_0^1\|{\rm d}_e\tau_{\delta_{st}e^v}[v]\|_{\delta_{st}e^v}\,{\rm d}s=\alpha t\|v\|_e,
	\end{split}\]
	where the last equality comes from the left invariance of the norm.
	This means that, in order to compute \(d_\psi(e,\delta_t e^v)\), it is
	sufficient to consider those horizontal curves \(\gamma\colon\G\to\mathbb R\)
	joining \(e\) to \(\delta_t e^v\) and satisfying
	\(\int_0^1\|\dot\gamma_s\|_{\gamma_s}\,{\rm d}s\leq
	\alpha\int_0^1\psi(\gamma_s,\dot\gamma_s)\,{\rm d}s\leq
	\alpha^2 t\|v\|_e\). We can also assume without loss of generality
	that any such curve \(\gamma\) is parametrized by constant speed
	with respect to the metric \(\|\cdot\|_x\). All in all, we have shown that
	\begin{equation}\label{eq:ineq_lsc_aux1}
	d_\psi(e,\delta_t e^v)=\inf_{\gamma\in\mathcal C_t}
	\int_0^1\psi(\gamma_s,\dot\gamma_s)\,{\rm d}s\quad\text{ for every }t>0,
	\end{equation}
	where the family \(\mathcal C_t\) of curves is defined as
	\[
	\mathcal C_t\coloneqq\bigg\{\gamma\colon[0,1]\to\G\text{ horizontal}
	\;\bigg|\;\gamma_0=e,\,\gamma_1=\delta_t e^v,
	\|\dot\gamma_s\|_{\gamma_s}\equiv\int_0^1\|\dot\gamma_s\|_{\gamma_s}
	\,{\rm d}s\leq\alpha^2 t\|v\|_e\bigg\}.
	\]
	Now fix any \(\varepsilon>0\). Since the map
	\({\rm exp}^{-1}\colon\G\to\g\) is a diffeomorphism, we can consider
	its differential \({\rm d}_x{\rm exp}^{-1}\colon T_x\G\to
	T_{{\rm exp}^{-1}(x)}\g\cong\g\) at any point \(x\in\G\).
	Let us observe that \({\rm exp}^{-1}\) is smooth, and
	\({\rm d}_e{\rm exp}^{-1}={\rm d}_e\tau_{e^{-1}}={\rm id}_\g\).
	
	Since \(\psi\) is lower semicontinuous and by the previous argument, we can find \(r>0\) such that
	\begin{subequations}\begin{align}\label{eq:ineq_lsc_aux2}
		\psi(x,v)\geq\psi(e,{\rm d}_x\tau_{x^{-1}}[v])-\varepsilon&\quad
		\text{ for every }x\in B(e,r)\text{ and }v\in H_x\G,\,\|v\|_x\leq 1,\\
		\label{eq:ineq_lsc_aux3}
		\big\|{\rm d}_x{\rm exp}^{-1}-{\rm d}_x\tau_{x^{-1}}\big\|
		_{{\rm L}(T_x\G,\g)}\leq\varepsilon&\quad\text{ for every }x\in B(e,r),
		\end{align}\end{subequations}
	where \(B(e,r) \equiv B_{d_{cc}}(e,r).\)
	In particular, given any \(t>0\) with
	\(\alpha^2 t\|v\|_e<r\) and \(\gamma\in\mathcal C_t\), we have that
	\(d_{cc}(e,\gamma_s)\leq s\alpha^2 t\|v\|_e<r\) for every \(s\in[0,1]\)
	and \(\|\dot\gamma_s\|_{\gamma_s}\leq\alpha^2 t\|v\|_e\)
	for a.e.\ \(s\in[0,1]\), thus accordingly \eqref{eq:ineq_lsc_aux2}
	and \eqref{eq:ineq_lsc_aux3} yield
	\begin{subequations}\begin{align}\label{eq:ineq_lsc_aux4}
		\psi(\gamma_s,\dot\gamma_s)\geq\psi(e,{\rm d}_{\gamma_s
		}\tau_{\gamma_s^{-1}}[\dot\gamma_s])-\alpha^2 t\|v\|_e\varepsilon&
		\quad\text{ for a.e.\ }s\in[0,1],\\
		\label{eq:ineq_lsc_aux5}
		\big\|{\rm d}_{\gamma_s}{\rm exp}^{-1}-{\rm d}_{\gamma_s}
		\tau_{\gamma_s^{-1}}\big\|_{{\rm L}(T_{\gamma_s}\G,\g)}\leq
		\varepsilon&\quad\text{ for a.e.\ }s\in[0,1],
		\end{align}\end{subequations}
	respectively. Therefore, for any \(t>0\) with
	\(\alpha^2 t\|v\|_e<r\) and \(\gamma\in\mathcal C_t\), we may estimate
	\[\begin{split}
	&\bigg|\psi\bigg(e,\int_0^1{\rm d}_{\gamma_s}\tau_{\gamma_s^{-1}}
	[\dot\gamma_s]\,{\rm d}s\bigg)-{\sf n}\bigg(\int_0^1{\rm d}_{\gamma_s}
	{\rm exp}^{-1}[\dot\gamma_s]\,{\rm d}s\bigg)\bigg|\\
	\overset{\phantom{\eqref{eq:ineq_lsc_aux5}}}\leq\,&{\sf n}\bigg(\int_0^1{\rm d}_{\gamma_s}\tau_{\gamma_s^{-1}}
	[\dot\gamma_s]\,{\rm d}s-\int_0^1{\rm d}_{\gamma_s}
	{\rm exp}^{-1}[\dot\gamma_s]\,{\rm d}s\bigg)
	\leq\lambda\,\bigg\|\int_0^1{\rm d}_{\gamma_s}\tau_{\gamma_s^{-1}}
	[\dot\gamma_s]-{\rm d}_{\gamma_s}{\rm exp}^{-1}[\dot\gamma_s]\,{\rm d}s\bigg\|_e\\
	\overset{\phantom{\eqref{eq:ineq_lsc_aux5}}}\leq\,&\lambda\int_0^1\big\|({\rm d}_{\gamma_s}\tau_{\gamma_s^{-1}}
	-{\rm d}_{\gamma_s}{\rm exp}^{-1})[\dot\gamma_s]
	\big\|_e\,{\rm d}s\leq\lambda\int_0^1\big\|{\rm d}_{\gamma_s}
	{\rm exp}^{-1}-{\rm d}_{\gamma_s}\tau_{\gamma_s^{-1}}\big\|_{{\rm L}(T_{\gamma_s}\G,\g)}\|\dot\gamma_s\|_{\gamma_s}\,{\rm d}s\\
	\overset{\eqref{eq:ineq_lsc_aux5}}\leq\,&\lambda\varepsilon
	\int_0^1\|\dot\gamma_s\|_{\gamma_s}\,{\rm d}s
	\leq\lambda\varepsilon\alpha^2 t\|v\|_e,
	\end{split}\]
	whence it follows that
	\begin{equation}\label{eq:ineq_lsc_aux6}\begin{split}
	\int_0^1\psi(\gamma_s,\dot\gamma_s)\,{\rm d}s&
	\overset{\eqref{eq:ineq_lsc_aux4}}\geq
	\int_0^1\psi(e,{\rm d}_{\gamma_s}\tau_{\gamma_s^{-1}}[\dot\gamma_s])
	\,{\rm d}s-\alpha^2 t\|v\|_e\varepsilon\\
	&\overset{\phantom{\eqref{eq:ineq_lsc_aux4}}}\geq
	\psi\bigg(e,\int_0^1{\rm d}_{\gamma_s}\tau_{\gamma_s^{-1}}
	[\dot\gamma_s]\,{\rm d}s\bigg)-\alpha^2 t\|v\|_e\varepsilon\\
	&\overset{\phantom{\eqref{eq:ineq_lsc_aux4}}}\geq
	{\sf n}\bigg(\int_0^1{\rm d}_{\gamma_s}{\rm exp}^{-1}[\dot\gamma_s]
	\,{\rm d}s\bigg)-(\lambda+1)\alpha^2 t\|v\|_e\varepsilon.
	\end{split}\end{equation}
	where in the second inequality we applied Jensen's inequality
	to \(\psi(e,\cdot)\). Now consider the curve \(\sigma\) in the
	Hilbert space \((\g,{\sf n})\), which is given by
	\(\sigma_s\coloneqq{\rm exp}^{-1}(\gamma_s)\) for every \(s\in[0,1]\).
	It holds that \(\sigma\) is absolutely continuous and satisfies
	\(\dot\sigma_s={\rm d}_{\gamma_s}{\rm exp}^{-1}[\dot\gamma_s]\)
	for a.e.\ \(s\in[0,1]\), thus
	\begin{equation}\label{eq:ineq_lsc_aux7}\begin{split}
	tv&=tv-0_\g={\rm exp}^{-1}(\delta_t e^v)-{\rm exp}^{-1}(e)=
	{\rm exp}^{-1}(\gamma_1)-{\rm exp}^{-1}(\gamma_0)=\sigma_1-\sigma_0\\
	&=\int_0^1\dot\sigma_s\,{\rm d}s=
	\int_0^1{\rm d}_{\gamma_s}{\rm exp}^{-1}[\dot\gamma_s]\,{\rm d}s.
	\end{split}\end{equation}
	By combining \eqref{eq:ineq_lsc_aux6} and \eqref{eq:ineq_lsc_aux7},
	we obtain for any \(t>0\) with \(\alpha^2 t\|v\|_e<r\) and
	\(\gamma\in\mathcal C_t\) that
	\begin{equation}\label{eq:ineq_lsc_aux8}
	\int_0^1\psi(\gamma_s,\dot\gamma_s)\,{\rm d}s
	\geq{\sf n}(tv)-(\lambda+1)\alpha^2 t\|v\|_e\varepsilon
	=\big[\psi(e,v)-(\lambda+1)\alpha^2\|v\|_e\varepsilon\big]t.
	\end{equation}
	We are now in a position to conclude the proof of the statement:
	given \(t>0\) with \(\alpha^2 t\|v\|_e<r\), one has that
	\begin{equation}\label{eq:ineq_lsc_aux9}
	\frac{d_\psi(e,\delta_t e^v)}{t}\overset{\eqref{eq:ineq_lsc_aux1}}=
	\inf_{\gamma\in\mathcal C_t}\frac{1}{t}\int_0^1\psi
	(\gamma_s,\dot\gamma_s)\,{\rm d}s\overset{\eqref{eq:ineq_lsc_aux8}}\geq
	\psi(e,v)-(\lambda+1)\alpha^2\|v\|_e\varepsilon.
	\end{equation}
	By letting \(t\searrow 0\), we thus deduce that
	\begin{equation}\label{eq:ineq_lsc_aux10}
	\varphi_{d_\psi}(e,v)=
	\limsup_{t\searrow 0}\frac{d_\psi(e,\delta_t e^v)}{t}
	\overset{\eqref{eq:ineq_lsc_aux9}}\geq
	\psi(e,v)-(\lambda+1)\alpha^2\|v\|_e\varepsilon.
	\end{equation}
	Finally, by letting \(\varepsilon\searrow 0\) in
	\eqref{eq:ineq_lsc_aux10} we conclude that
	\(\varphi_{d_\psi}(e,v)\geq\psi(e,v)\), as desired.
\endproof
\begin{cor}
If \(\psi\) is a continuous sub-Finsler convex metric, then    \[
    \varphi_{d_\psi}(x,v)=\psi(x,v)\quad\text{ for every }(x,v)\in H\G. \]
\end{cor}
\begin{proof}
It is an immediate consequence of assertions ii) and iii) of Theorem \ref{uguaglianza}.
\end{proof}

The crucial observation below states that $\delta_{\varphi}$ coincides with the intrinsic distance $d_{\varphi^\star}$ when we assume that the sub-Finsler metric is lower semicontinuous.
This will allow us to show the same result when $\varphi$ is upper semicontinuous, thanks to an approximation argument.

\begin{thm}\label{maintheorem} 
Let \(\varphi\in\mathcal M_{cc}^{\alpha}(\mathbb G)\) be a sub-Finsler convex metric.
Then it holds that \(\delta_\varphi\leq d_{\varphi^\star}\).
Moreover, if \(\varphi\) is lower semicontinuous, then
 $$\delta_\varphi(x, y)=d_{\varphi^\star}(x, y) \quad\text{for every } x, y \in \mathbb{G}.$$
\end{thm}
\begin{proof} Let \(x,y\in\mathbb G\) be fixed. To prove the first
part of the statement, pick any Lipschitz function $f$
with $\|\varphi(z,\nabla_\G f(z))\|_\infty\leq 1$
and any horizontal curve \(\gamma : [0,1] \rightarrow \mathbb{G}\) joining $x$ and $y$ such that
 $$\mathcal{H}^1\big(\gamma \cap \{ z \in \G \;: \; \varphi(z, \nabla_{\mathbb{G}} f(z)) > 1\}\big) = 0.$$
These are competitors for
\(\delta_\varphi(x,y)\) and \(d_{\varphi^\star}(x,y)\),
respectively. Then we can estimate
\begin{align*}
\big|f(x)-f(y)\big|&=\bigg|\int_0^1\frac{\rm d}{\dt}(f(\gamma(t)))\,{\rm d}t \bigg|
=\bigg|\int_0^1\langle\nabla_{\mathbb G}f(\gamma(t)),
\dot\gamma(t) \rangle_{\gamma(t)} \,{\rm d}t\bigg| \\
&\leq
\int_0^1\big|\langle\nabla_{\mathbb G}f(\gamma(t)),\dot\gamma(t)
\rangle_{\gamma(t)}\big|\,{\rm d}t
\leq\int_0^1\varphi(\gamma(t),\nabla_{\mathbb G}f(\gamma(t)))\,
\varphi^\star(\gamma(t),\dot\gamma(t)) \,{\rm d}t \\
&\leq\big\|\varphi(\cdot,\nabla_{\mathbb G}f(\cdot))\big\|_\infty
\int_0^1\varphi^\star(\gamma(t),\dot\gamma(t)) \,{\rm d}t
\leq\int_0^1\varphi^\star(\gamma(t),\dot\gamma(t)) \,{\rm d}t,
\end{align*}
whence it follows that \(\delta_\varphi(x,y)\leq d_{\varphi^\star}(x,y)\).

Now suppose \(\varphi\) is lower semicontinuous. Define the function
$f : \mathbb{G} \rightarrow \mathbb{R}$ as $ f( \cdot) \coloneqq d_{\varphi^\star}(x, \cdot)$ and since
\(d_{\varphi^\star}(x, y) \leq\alpha^{-1}d_{cc}(x, y)\) everywhere, we have that \(f\) is Lipschitz.
Fix any point \(z\in\mathbb G\) such that \(\nabla_{\mathbb G}f(z)\) exists
and let \(v\in H_z\mathbb G\).
Pick a horizontal curve \(\gamma\colon[0,\varepsilon]\to\mathbb G\) of class
\(C^1\) such that \(\gamma(0)=z\) and \(\dot\gamma(0)=v\). Thanks to the
continuity of \(t\mapsto(\gamma(t),\dot\gamma(t))\) and the upper
semicontinuity of \(\varphi^\star\), granted by Lemma 
\ref{semic}, we obtain that \(\limsup_{t\searrow 0}
\fint_0^t\varphi^\star(\gamma(s),\dot\gamma(s))\, \,{\rm d}s\leq
\varphi^\star(\gamma(0),\dot\gamma(0))=\varphi^\star(z,v)\),
whence, by the identities \eqref{proj} and \eqref{diff2}, it follows that
\[\begin{split}
\langle\nabla_{\mathbb G}f(z),v\rangle_z&=
\lim_{t\searrow 0}\frac{f(\gamma(t))-f(\gamma(0))}
{t}\leq\limsup_{t\searrow 0}\frac{f(\gamma(t))-f(\gamma(0))}
{d_{\varphi^\star}(\gamma(t),\gamma(0))}\limsup_{t\searrow 0}
\frac{d_{\varphi^\star}(\gamma(t),\gamma(0))}{t}\\
&\leq\limsup_{t\searrow 0}\frac{\big|d_{\varphi^\star}(x,\gamma(t))-
d_{\varphi^\star}(x,\gamma(0))\big|}{d_{\varphi^\star}(\gamma(t),\gamma(0))}
\limsup_{t\searrow 0}\fint_0^t\varphi^\star(\gamma(s),\dot\gamma(s))\,{\rm d}s
\leq\varphi^\star(z,v).
\end{split}\]
By arbitrariness of \(v\in H_z\mathbb G\), we deduce that
\(\varphi(z,\nabla_{\mathbb G}f(z))\leq 1\). Therefore, \(f\) is a competitor for \(\delta_\varphi(x,y)\). This implies
that \(\delta_\varphi(x,y)\geq \lvert f(x)-f(y) \rvert=d_{\varphi^\star}(x,y)\).
\end{proof}

In particular, the last part of the proof shows that the supremum appearing in the definition of \(\delta_\varphi(x,y)\) is actually a maximum.

The upper semicontinuity of the sub-Finsler metric $\varphi$ is crucial for our proof, because it allows us to approximate the dual metric $\varphi^{\star}$ through a family of continuous Finsler metrics. 

\begin{cor}\label{corollary}
Let \(\varphi\in\mathcal M_{cc}^{\alpha}(\mathbb G)\) be a sub-Finsler convex metric.
Suppose \(\varphi\) is upper semicontinuous. Then, for every $x, y \in \mathbb{G}$ it holds that
\(\delta_\varphi(x, y)=d_{\varphi^\star}(x, y)\).
\end{cor}
\begin{proof}
Lemma \ref{semic} ensures that \(\varphi^\star\) is lower semicontinuous. We set
\(\widetilde\varphi^\star\colon T\mathbb G\to[0,+\infty)\) as
\[
\widetilde\varphi^\star(x,v)\coloneqq\left\{\begin{array}{ll}
\varphi^\star(x,v),\\
+\infty,
\end{array}\quad\begin{array}{ll}
\text{ if }(x,v)\in H\mathbb G,\\
\text{ if }(x,v)\in T\mathbb G\setminus H\mathbb G.
\end{array}\right.
\]
Observe that \(H\G\) is closed in \(T\G\) and thus \(\widetilde\varphi^\star\) is lower semicontinuous. Thanks to
\cite[Theorem 3.11]{LDLP}, there exists a sequence \( F_n : T \mathbb{G} \rightarrow [0,+\infty)\)
of Finsler metrics on \(\mathbb G\) such that
\(F_n(x,v)\nearrow\tilde\varphi^\star(x,v)\) for every
\((x,v)\in T\mathbb G\). 
Setting 
$$\varphi_n\colon H\mathbb{G}\to[0+\infty)
\quad\text{as } \varphi_n\coloneqq(F_n|_{H\mathbb G})^\star,$$
we obtain that $\varphi_n\in\mathcal M_{cc}^\alpha(\mathbb G)$
and \(\varphi_n^\star(x,v)\nearrow\varphi^\star(x,v)\)
for every \((x,v)\in H\G\).
Therefore \(\varphi_n(x,v)\searrow\varphi(x,v)\) for every \((x,v)\in H\mathbb G\).
In particular, the inequality \(\varphi_n\geq\varphi\) holds
for all \(n\in\mathbb N\). This implies that any competitor \(f\) 
for \(\delta_{\varphi_n}\) is a competitor
for \(\delta_\varphi\), so that accordingly 
\begin{equation}\label{firstineq}
\delta_{\varphi_n}(x, y) \leq\delta_\varphi(x, y),\quad\text{ for every } n\in\mathbb{N}\text{ and } x, y \in \mathbb{G}.
\end{equation}
Moreover, since  the infimum in the definition of $d_{F_n}$ is computed with respect to all Lipschitz curves, while the infimum in the definition of $d_{\varphi_n^\star}$ is just over horizontal curves, for every $x, y \in \mathbb{G}$ we get that
\begin{equation}\label{secondineq}
d_{F_n}(x, y)\leq d_{\varphi_n^\star}(x, y)\leq d_{\varphi^\star}(x, y)
\quad\text{ for every }n\in\mathbb{N}.
\end{equation}
From the convergence of \(F_n \) to \(\widetilde\varphi^\star\) we deduce 
that \(d_{F_n}(x,y)\rightarrow d_{\varphi^\star}(x,y)\) for every \(x,y\in\mathbb{G}\)
(cf. the proof of \cite[Theorem 5.1]{LDLP}), and thus
\begin{equation}\label{thirdineq}
d_{\varphi^\star}(x,y)=\lim_{n\to\infty}d_{\varphi_n^\star}(x,y)
\quad\text{ for every }x,y\in\mathbb G.
\end{equation}
Finally, since \(\varphi_n\) is lower semicontinuous (actually, continuous)
by Lemma \ref{semic}, we know from the second part of Theorem \ref{maintheorem} that
\begin{equation}\label{fourineq}
\delta_{\varphi_n}(x, y)=d_{\varphi_n^\star}(x, y)\quad\text{for every }n\in\mathbb{N}.
\end{equation}
All in all, we obtain that
\[
d_{\varphi^\star}(x,y)\overset{\eqref{thirdineq}}=
\lim_{n\to\infty}d_{\varphi_n^\star}(x,y)\overset{\eqref{fourineq}}=
\lim_{n\to\infty}\delta_{\varphi_n}(x,y)\overset{\eqref{firstineq}}\leq
\delta_\varphi(x,y)\quad\text{ for every }x,y\in\mathbb G.
\]
Since the converse inequality \(d_{\varphi^\star}\geq\delta_\varphi\)
is granted by the first part of Theorem \ref{maintheorem}, we
conclude that \(\delta_\varphi=d_{\varphi^\star}\), as required.
\end{proof}

\begin{thm}\label{phi=Lip} Let $\varphi \in \mathcal{M}_{cc}^\alpha(\mathbb{G})$ be an upper semicontinuous sub-Finsler convex metric. Then for any locally Lipschitz function $f : \mathbb{G} \rightarrow \mathbb{R}$ we have that
$$ \varphi(x, \nabla_{\mathbb{G}} f(x))= \Lip_{\delta_{\varphi}} f(x) \quad\text{for a.e.\ }  x \in \mathbb{G} .$$
\end{thm}
\begin{proof}
\ \\
\boxed{\leq} Since both sides are positively $1$-homogeneous with respect to $f$, we only need to show
that, if $\Lip_{\delta_{\varphi}} f(x) = 1$, then 
$\varphi(x, \nabla_{\mathbb{G}} f(x)) \leq 1$ for a.e.\ $x \in \mathbb{G}$. By Corollary \ref{corollary}, 
$ \Lip_{\delta_{\varphi}} f(x) = \Lip_{d_{\varphi^\star}} f(x)$, hence if we fix $(x, v) \in H \mathbb{G}$, thanks to
\eqref{proj} and the expression \eqref{diff2} we can write:
\begin{align*}
\langle \nabla_{\mathbb{G}} f(x), v \rangle_x &
= \lim_{t \rightarrow 0} \frac{f(x \cdot \delta_t e^{\bar{v}}) - f(x)}{t}
\leq \limsup_{t \rightarrow 0} \frac{d_{\varphi^\star}(x, x \cdot \delta_t e^{\bar{v}})}{t} \cdot \limsup_{t \rightarrow 0}
\frac{\lvert f(x \cdot \delta_t e^{\bar{v}}) - f(x) \rvert}{d_{\varphi^\star}(x, x \cdot \delta_t e^{\bar{v}})} \\
&\leq \varphi_{d_{\varphi^\star}}(x, v) \Lip_{d_{\varphi^\star}}f(x) \leq \varphi^{\star}(x, v),
\end{align*}
where in the last inequality we used item i) of Theorem \ref{uguaglianza}.
By arbitrariness of $v \in H_x \mathbb{G}$ and the fact that 
$$  \varphi(x, \nabla_{\mathbb{G}} f(x)) = \varphi^{\star \star}(x, \nabla_{\mathbb{G}} f(x)) \leq 1, $$
we get the conclusion.\\
\boxed{\geq} Thanks to a convolution argument, we can find a sequence $ (f_n)_{n} \subset C^1(\mathbb{G})$ such that 
$f_n \rightarrow f$ uniformly on compact sets and $\nabla_{\mathbb{G}} f_n \rightarrow \nabla_{\mathbb{G}}f$ in the almost
everywhere sense. Recall that any $C^1$-function is locally Lipschitz. Fix any $x \in \mathbb{G}$ such that 
$\nabla_{\mathbb{G}} f_n(x)$ exists for all \(n\in\mathbb N\) and $\nabla_{\mathbb{G}} f_n(x) \rightarrow \nabla_{\mathbb{G}}f(x)$
as $n \rightarrow \infty$. Now let $\varepsilon > 0$ be fixed. Then we can choose $r' > 0$ and $\bar{n} \in \mathbb{N}$
so that 
$$ \sup_{B(x, 2 r')} \lvert f_{\bar n} - f \rvert \leq \varepsilon \quad  \text{and} \quad \varphi(x, \nabla_{\mathbb{G}} f_{\bar n}(x) - \nabla_{\mathbb{G}}f(x)) \leq \varepsilon,$$ 
where the ball is with respect to the distance $d^\star_{\varphi}$.
Calling $g \coloneqq f_{ \bar{n}}$ and being $z \mapsto \nabla_{\mathbb{G}} g(z)$ continuous, we deduce that 
$z \mapsto \varphi(z, \nabla_{\mathbb{G}}g(z))$ is upper semicontinuous, thus there exists $r < r'$ such that
$$
\varphi(y, \nabla_{\mathbb{G}}g(y)) \leq \varphi(x, \nabla_{\mathbb{G}}g(x)) + \varepsilon
 \quad\text{for every }y \in B(x, 2 r).
$$
Fix any point $y \in B(x, r)$ and consider a horizontal curve $\gamma : [0,1] \rightarrow \mathbb{G}$ such that
$\gamma(0) = x$, $\gamma(1) = y$ \ with \ $\gamma([0,1]) \subset B(x, 2 r)$. We can estimate in this way:
\begin{align*}
\lvert f(x) - f(y)\rvert &\leq \lvert g(x) - g(y)\rvert + 2 \varepsilon \leq \int_0^1 \frac{\rm d}{\dt} g(\gamma(t)) \,{\rm d}t + 2 \varepsilon \\
& \leq \int_0^1 \varphi(\gamma(t), \nabla_\mathbb{G} g(\gamma(t))) \varphi^{\star}(\gamma(t), \dot{\gamma}(t)) \,{\rm d}t + 2 \varepsilon \\
& \leq \Bigl( \varphi(x, \nabla_{\mathbb{G}} g(x)) + \varepsilon \Bigr) \int_0^1 \varphi^{\star}(\gamma(t), \dot{\gamma}(t))\,{\rm d}t 
+ 2 \varepsilon \\
& \leq \Bigl( \varphi(x, \nabla_{\mathbb{G}} f(x)) + 2 \varepsilon \Bigr) \int_0^1 \varphi^{\star}(\gamma(t), \dot{\gamma}(t))\,{\rm d}t
+ 2 \varepsilon.
\end{align*}
By taking the infimum over all $\gamma \in \mathcal{H}([0,1], B(x, 2 r))$, we obtain that
$$ \lvert f(x) - f(y) \rvert \leq \Bigl(\varphi(x, \nabla_{\mathbb{G}}f(x)) + 2 \varepsilon \Bigr) d_{\varphi^\star}(x, y) +
2 \varepsilon,$$
whence by letting $\varepsilon \rightarrow 0$ we obtain that
$$ \frac{\lvert f(x) - f(y) \rvert}{d_{\varphi^\star}(x, y)} \leq \varphi(x, \nabla_{\mathbb{G}}f(x)). $$
Finally, by letting $y \rightarrow x$ we conclude that
$$ \Lip_{\delta_{\varphi}} f(x)= \Lip_{d_{\varphi^\star}} f(x) \leq \varphi(x, \nabla_{\mathbb{G}}f(x)),$$
as required.
\end{proof}
To conclude, in Proposition \ref{smooth} we prove that in the definition \eqref{lipdist}
of the distance \(\delta_\varphi\) it is sufficient to consider smooth functions. 
Before passing to the proof of this claim,
we prove the following technical result.
By \({\rm LIP}_{d^\star_\varphi}(f)\in[0,+\infty)\)
we mean the (global) Lipschitz constant of
\(f\in{\rm LIP}_{d^\star_\varphi}(\G)\). 
\begin{lem}\label{lem:sup_lip}
Let \(\varphi\in\mathcal M_{cc}^{\alpha}(\mathbb G)\) be a sub-Finsler convex metric.
Then it holds that
\begin{equation}\label{eq:sup_lip}
{\rm LIP}_{d_{\varphi^\star}}(f)=\underset{x\in\mathbb G}{\rm ess\,sup}\,{\rm Lip}
_{d_{\varphi^\star}}f(x)\quad\text{ for every }f\in{\rm LIP}_{d_{\varphi^\star}}(\mathbb G).
\end{equation}
\end{lem}
\begin{proof}
The inequality \((\geq)\) is trivial. To prove the converse inequality,
we argue by contradiction: suppose there exist \(x,y\in\mathbb G\)
with \(x\neq y\), a negligible Borel set \(N\subseteq\mathbb G\)
and \(\delta>0\) such that
\[
\frac{\big|f(x)-f(y)\big|}{d_{\varphi^\star}(x,y)}\geq
\sup_{z\in\mathbb G\setminus N}{\rm Lip}_{d_{\varphi^\star}}f(z)+\delta.
\]
Given any \(\varepsilon>0\), we can find \(\gamma \in \mathcal{H}([0,1], \mathbb{G})\)\, such that \(\gamma(0)=x\), \(\gamma(1)=y\), and
\[
\int_0^1\varphi^*(\gamma(t),\dot\gamma(t))\,{\rm d}t\leq
d_{\varphi^\star}(x,y)+\varepsilon.
\]
Since $f \circ \gamma : \mathbb{R} \rightarrow \mathbb{R}$
is Lipschitz, so Pansu-differentiable almost everywhere, we deduce that
\[\begin{split}
\big|f(x)-f(y)\big|&\leq\int_0^1\big|(f\circ\gamma)'(t)\big|\,{\rm d}t
\leq\int_0^1\varphi(\gamma(t),\nabla_{\mathbb G}f(\gamma(t)))\,
\varphi^*(\gamma(t),\dot\gamma(t))\,{\rm d}t\\
&=\int_0^1{\rm Lip}_{d_{\varphi^\star}}f(\gamma(t))\,
\varphi^\star(\gamma(t),\dot\gamma(t))\,{\rm d}t\leq
\sup_{z\in\mathbb G\setminus N}{\rm Lip}_{d_{\varphi^\star}}f(z)
\int_0^1\varphi^\star(\gamma(t),\dot\gamma(t))\,{\rm d}t\\
&\leq\bigg[\frac{\big|f(x)-f(y)\big|}{d_{\varphi^\star}(x,y)}-\delta\bigg]
\big(d_{\varphi^\star}(x,y)+\varepsilon\big).
\end{split}\]
By letting \(\varepsilon\searrow 0\) in the above estimate, we get
\(0\leq -\delta\,d_{\varphi^\star}(x,y)\), which leads to a contradiction.
Therefore, also the inequality \((\leq)\) in \eqref{eq:sup_lip} is proved,
whence the statement follows.
\end{proof}
\begin{prop}\label{smooth}
Let \(\varphi\in\mathcal M_{cc}^{\alpha}(\mathbb G)\) be a sub-Finsler convex metric.
Suppose \(\varphi\) is upper semicontinuous. Then for any \(x,y\in\mathbb G\)
it holds that
\begin{equation}\label{eq:alt_form_delta}
\delta_\varphi(x,y)=\sup\Big\{\big|f(x)-f(y)\big|\;\Big|\;
f\in C^\infty(\mathbb G),\,\big\|\varphi(\cdot,\nabla_{\mathbb G}f(\cdot)
\big\|_\infty\leq 1\Big\}.
\end{equation}
\end{prop}
\begin{proof}
Denote by \(\tilde\delta_\varphi(x,y)\) the quantity in the
right-hand side of \eqref{eq:alt_form_delta}. Since any competitor for
\(\tilde\delta_\varphi(x,y)\) is a competitor for \(\delta_\varphi(x,y)\),
we have that \(\delta_\varphi(x,y)\geq\tilde\delta_\varphi(x,y)\).
To prove the converse inequality, fix any Lipschitz function
\(f\colon\mathbb G\to\mathbb R\) such that
\(\big\|\varphi(\cdot,\nabla_{\mathbb G}f(\cdot))\big\|_\infty\leq 1\).
Corollary \ref{corollary} and Theorem \ref{phi=Lip} grant that
\({\rm ess\,sup}\,{\rm Lip}_{d_{\varphi^\star}}f\leq 1\), thus
Lemma \ref{lem:sup_lip} yields \({\rm LIP}_{d_{\varphi^\star}}(f)\leq 1\).
Given that \(d_{\varphi^\star}\) is an increasing, pointwise limit of Finsler distances by \cite[Theorem 3.11]{LDLP}, we are in a position to apply Theorem
\ref{thm:smooth_approx}. Thus we obtain a sequence \((f_n)_n\subseteq
C^\infty(\mathbb G)\cap{\rm LIP}_{d_{\varphi^\star}}(\mathbb G)\) such that
\({\rm LIP}_{d_{\varphi^\star}}(f_n)\leq 1\) for all \(n\in\mathbb N\) and
\(f_n\to f\) uniformly on compact sets. Corollary \ref{corollary}
and Theorem \ref{phi=Lip} imply that
\(\big\|\varphi(\cdot,\nabla_{\mathbb G}f_n(\cdot))\big\|_\infty=
\sup{\rm Lip}_{d_{\varphi^\star}}f_n\leq 1\), thus \(f_n\) is a competitor
for \(\tilde\delta_\varphi(x,y)\). Then we conclude that
\(\big|f(x)-f(y)\big|=\lim_n\big|f_n(x)-f_n(y)\big|\leq
\tilde\delta_\varphi(x,y)\), whence it follows that
\(\delta_\varphi(x,y)\leq\tilde\delta_\varphi(x,y)\) by arbitrariness
of \(f\).
\end{proof}

\appendix

\section{Smooth approximation of Lipschitz functions on generalized sub-Finsler manifolds}
The aim of this appendix is to prove an approximation result for
real-valued Lipschitz functions defined on some very weak kind
of sub-Finsler manifold. More precisely, we consider a distance
\(d\) on a smooth manifold that can be obtained as the monotone
increasing limit of Finsler distances; this notion covers the
case of generalized (so, possibly rank--varying) sub-Finsler
manifolds, thanks to \cite[Theorem 3.11]{LDLP}. In this
framework, we prove (see Theorem \ref{thm:smooth_approx} below)
that any Lipschitz function can be approximated (uniformly
on compact sets) by smooth functions having the same Lipschitz
constant. This generalizes previous results that were known
on `classical' sub-Riemannian manifolds,
cf.\ \cite{HK00} and the references therein.
\bigskip

Let us fix some notation. Given a metric space \((X,d)\),
we denote by \({\rm LIP}_d(X)\) the family of real-valued
Lipschitz functions on \(X\). For any
\(f\in{\rm LIP}_d(X)\), we denote by \({\rm LIP}_d(f)\in[0,+\infty)\) and
\({\rm Lip}_d f\colon X\to[0,+\infty)\) the (global) Lipschitz constant and the pointwise Lipschitz constant of \(f\), respectively. Moreover, given a Finsler manifold \((M,F)\),
we denote by \(d_F\)  the length distance on \(M\) induced by
the Finsler metric \(F\).
\begin{thm}\label{thm:smooth_approx}
Let \(M\) be a smooth manifold. Let \(d\) be a distance on \(M\) having
the following property: there exists a sequence \((F_i)_i\) of Finsler
metrics on \(M\) such that
\[
d_{F_i}(x,y)\nearrow d(x,y)\quad\text{ for every }x,y\in M.
\]
Then for any \(f\in{\rm LIP}_d(M)\) there exists a sequence
\((f_n)_n\subseteq C^\infty(M)\cap{\rm LIP}_d(M)\) such that
\[
\sup_{n\in\mathbb N}{\rm LIP}_d(f_n)\leq{\rm LIP}_d(f),\qquad
f_n\to f\;\text{ uniformly on compact sets.}
\]
\end{thm}
\begin{proof}
Denote \(L\coloneqq{\rm LIP}_d(f)\) and \(d_i\coloneqq d_{F_i}\) for
every \(i\in\mathbb N\). Choose any countable, dense subset \((x_j)_j\) of \((M,d)\).
Given any \(n\in\mathbb N\), we define the function \(h_n\in{\rm LIP}_d(M)\) as
\[
h_n(x)\coloneqq\big(-L\,d(x,x_1)+f(x_1)\big)\vee\cdots\vee
\big(-L\,d(x,x_n)+f(x_n)\big)-\frac{1}{n}\quad\text{ for every }x\in M.
\]
Observe that \({\rm LIP}_d(h_n)\leq L\) and that \(h_n(x)<h_{n+1}(x)<f(x)\)
for every \(n\in\mathbb N\) and \(x\in M\). We claim that \(h_n(x)\nearrow f(x)\)
for all \(x\in M\). In order to prove it, fix any \(x\in M\) and \(\varepsilon>0\).
Pick some \(\bar n\in\mathbb N\) such that \(1/\bar n<\varepsilon\) and
\(d(x,x_{\bar n})<\varepsilon\). Then for every \(n\geq\bar n\) it holds that
\[
h_n(x)\geq -L\,d(x,x_{\bar n})+f(x_{\bar n})-\frac{1}{n}
\geq -L\,\varepsilon+\big(f(x)-L\,d(x,x_{\bar n})\big)-\frac{1}{\bar n}
\geq f(x)-(2L+1)\varepsilon,
\]
thus proving the claim. Fix an increasing sequence
\((K_n)_n\) of compact sets in \(M\) satisfying the following property:
given any compact set \(K\subseteq M\), there exists \(n\in\mathbb N\)
such that \(K\subseteq K_n\). In particular, one has that \(\bigcup_n K_n=M\).
Notice that \(h_n+\frac{1}{n(n+1)}\leq h_{n+1}\) on \(K_n\) for all \(n\in\mathbb N\).
Since \(d_{F_i}\nearrow d$, there exists \(i_n\in\mathbb N\) such that the function
\(g_n\colon M\to\mathbb R\), given by
\[
g_n(x)\coloneqq\big(-L\,d_{i_n}(x,x_1)+f(x_1)\big)\vee\cdots
\vee\big(-L\,d_{i_n}(x,x_n)+f(x_n)\big)-\frac{1}{n}\quad\text{ for every }x\in M,
\]
satisfies \(h_n<g_n<h_n+\frac{1}{n(n+1)}\) on \(K_n\). Note that
\(g_n\in{\rm LIP}_{d_{i_n}}(M)\) and \({\rm LIP}_{d_{i_n}}(g_n)=L\).
Thanks to a mollification argument, it is possible to build a function
\(f_n\in C^\infty(M)\cap{\rm LIP}_{d_{i_n}}(M)\) such that
\({\rm LIP}_{d_{i_n}}(f_n)\leq L\) and \(g_n<f_n<g_{n+1}\)
on \(K_n\). Therefore, for any \(n\in\mathbb N\) and \(x\in K_n\) it holds that
the sequence \(\big(f_j(x)\big)_{j\geq n}\) is strictly increasing and converging
to \(f(x)\). This grants that \(f_j\to f\) uniformly on \(K_n\) for any given \(n\in\mathbb N\).
Hence, our specific choice of \((K_n)_n\) implies that \(f_n\to f\) uniformly on
compact sets. Finally, the inequality \(d_{i_n}\leq d\) yields
\(f_n\in{\rm LIP}_d(M)\) and \({\rm LIP}_d(f_n)\leq{\rm LIP}_{d_{i_n}}(f_n)\leq L\)
for all \(n\in\mathbb N\), whence the statement follows.
\end{proof}

\def\cprime{$'$} \def\cprime{$'$}

\end{document}